\documentclass[11pt,a4paper,english,reqno]{amsart}
\usepackage{amsmath,amssymb,amsfonts,epsfig,mathrsfs}
\usepackage[T1]{fontenc}

\usepackage{color}
\usepackage{array}
\usepackage{amsthm}
\usepackage{amstext}
\usepackage{graphicx}
\usepackage{setspace}
\usepackage{setspace}

\usepackage[margin=2.5cm]{geometry}
\usepackage{bbm}
\usepackage{color}
\usepackage{enumitem}
\usepackage{undertilde}
\allowdisplaybreaks[4]

\usepackage{amscd,psfrag}
\usepackage{yhmath}

\usepackage{slashed}

\makeatletter
\pdfpageheight\paperheight
\pdfpagewidth\paperwidth

\usepackage{epstopdf}
\usepackage{chngcntr}
\counterwithin{figure}{section}
\usepackage{mathrsfs}

\usepackage{indentfirst}	

\usepackage[normalem]{ulem}
\theoremstyle{plain}

\newtheorem{definition}{Definition}[section]

\newtheorem*{theorem*}{Theorem}

\newtheorem{remark}[definition]{Remark}

\newtheorem*{remark*}{Remark}
\newtheorem*{sideremark*}{Side Remark}\newtheorem*{mt*}{Main Theorem}

\newtheorem*{claim*}{Claim}
\newtheorem*{q*}{Question}
\newtheorem{lemma}[definition]{Lemma}
\newtheorem{corollary}[definition]{Corollary}
\newtheorem*{corollary*}{Corollary}
\newtheorem*{proposition*}{Proposition}

\newtheorem{proposition}[definition]{Proposition}

\newcommand{\R}{\mathbb{R}}

\newcommand{\na}{\nabla}
\newcommand{\dd}{{\rm d}}
\newcommand{\p}{\partial}
\newcommand{\e}{\epsilon}
\newcommand{\emb}{\hookrightarrow}

\newcommand{\G}{\Gamma}
\newcommand{\M}{\mathcal{M}}

\newcommand{\1}{\mathbbm{1}}

\newcommand{\db}{\Delta^{\mathsf{B}}}

\newcommand{\ric}{\mathscr{R}}

\newcommand{\A}{{\mathfrak{A}}}
\newcommand{\fs}{{\mathfrak{S}}}
\newcommand{\TA}{{\widehat{\A}}}
\newcommand{\B}{{\mathfrak{B}}}
\newcommand{\TB}{{\widehat{\B}}}
\newcommand{\CC}{{\mathfrak{C}}}
\newcommand{\TC}{{\widehat{\CC}}}

\newcommand{\dvg}{{\dd V_{g_0}}}

\newcommand{\one}{{]1,\infty[}}

\newcommand{\gc}{{\overline{\mathfrak{g}_0}}}

\newcommand{\gn}{{\gamma_{\bf n}}}

\newcommand{\gt}{{\gamma_{\bf t}}}

\newcommand{\nnn}{{\mathfrak{n}}}
\newcommand{\ttt}{{\mathfrak{t}}}

\newcommand{\eu}{{\mathfrak{e}}}

\newcommand{\clM}{{\overline{\M}}}

\newcommand{\ok}{{\mathcal{O}(\kappa)}}

\newcommand{\half}{{\R^n_+}}

\allowdisplaybreaks[4]

\def\XXint#1#2#3{{\setbox0=\hbox{$#1{#2#3}{\int}$ }
\vcenter{\hbox{$#2#3$ }}\kern-.6\wd0}}

\newcommand{\mres}{\mathbin{\vrule height 1.6ex depth 0pt width
0.13ex\vrule height 0.13ex depth 0pt width 1.3ex}}

\numberwithin{equation}{section}
\numberwithin{figure}{section}

\title{A new proof of the Gaffney's inequality for differential forms on manifolds-with-boundary: the variational approach \`{a} la Kozono--Yanagisawa}

\author{Siran Li}\address{Siran Li: New York University--Shanghai, Office 1146, 1555 Century Avenue, Pudong District, Shanghai, China}
\email{\texttt{sl4025@nyu.edu}}



\date{\today}

\begin{document}

\maketitle

\begin{abstract}
Let $(\M,g_0)$ be a compact  Riemannian manifold-with-boundary. We present a new proof of the classical Gaffney's inequality for differential forms in boundary value spaces over $\M$, via the variational approach \`{a} la Kozono--Yanagisawa [$L^r$-variational inequality for vector fields and the Helmholtz--Weyl decomposition in bounded domains, \textit{Indiana Univ. Math. J.} \textbf{58} (2009), 1853--1920] combined with global computations based on the Bochner's  technique.
\end{abstract}

\section{Introduction}

The goal of this note is to give an alternative proof --- which is global and variational in nature --- of the Gaffney's inequality (\cite{g}) for differential forms of arbitrary order on a manifold-with-boundary $(\M^n,g_0)$ in boundary value Sobolev spaces $W^{1,r}$ with $r \in ]1,\infty[$.

\subsection{Gaffney's inequality}

As a prototype of the Gaffney's inequality, consider the well-known div-curl estimate of the Calder\'{o}n--Zygmund type in dimension $3$ --- on a 3D Euclidean domain (with mild regularity assumptions), the $W^{1,r}$-norm of a vectorfield $v$ can be estimated by the $L^r$-norms of ${\rm div} \,v$, ${\rm curl}\, v$, and $v$ itself in the interior of the domain. Such an estimate plays a fundamental role in mathematical hydrodynamics, magnetics, and many other fields.

Nonetheless, the extension of the classical div-curl estimate to higher-dimensional manifolds for differential forms of arbitrary order subject to suitable boundary conditions is by no means straightforward. In the seminal work \cite{g}, Gaffney first established the estimate 
\begin{equation}\label{gaffney}
\|u\|_{W^{1,2}(\M)} \leq C\left\{\|du\|_{L^2(\M)} + \|\delta u\|_{L^2(\M)} + \|u\|_{L^2(\M)}\right\},
\end{equation}
where $\M$ has empty boundary and $u$ is a differential $k$-form on $\M$ for any $k$. See also Gallot--Meyer \cite{gm} and the exposition in Schwarz \cite{s} (especially Corollary~2.1.6 and Theorem~2.1.7).

The proof of Eq.~\eqref{gaffney} for general $r \in \one$ (\emph{i.e.}, with $W^{1,2}$ and $L^2$ replaced by $W^{1,r}$ and $L^r$, respectively) in boundary value spaces is more technical; indeed, to the best of our knowledge, a geometric proof on manifolds-with-boundary remains elusive. The only proof carefully written down that we know of is in Iwaniec--Scott--Stroffolini \cite{iss}, \S 4. Utilising the ``freezing coefficient'' argument and thus reducing to the Euclidean case, the proof in \cite{iss} is essentially local and Euclidean in nature. On the other hand, it should be remarked that global proofs of Eq.~\eqref{gaffney} on manifolds without boundary can be found in the literature; \emph{cf.} Scott \cite{scott} for a potential-theoretic proof via the $L^r\to L^r$ boundedness of Riesz transforms. Extensions to various classes of open manifolds are also available (Amar \cite{a1}, Auscher--Coulhon--Duong--Hofmann \cite{acdh}, etc.).

We point out that in the Hilbert case ($r=2$), global, geometric, and simple proofs of the Gaffney's inequality on Riemannian manifolds and Euclidean domains can be found in the literature. Wider classes of boundary conditions and the determination of best constants have also been investigated. See Cast\'{o} \cite{c3}, Cast\'{o}--Dacorogna \cite{cd}, Cast\'{o}--Dacorogna--Kneuss \cite{cdk}, Cast\'{o}--Dacorogna--Sil \cite{cds}, Cast\'{o}--Dacorogna--Rajendran \cite{ckr}, and Mitrea \cite{mitrea}.

In this note we  present an alternative, \emph{global} proof of the $L^r$-Gaffney's inequality for differential forms of arbitrary order on manifolds-with-boundary. Our arguments utilise the Bochner technique and integration by parts (\emph{i.e.}, the Stokes' and/or Gauss--Green theorems), together with a variational characterisation of $W^{1,r}$-differential forms in boundary value spaces \emph{\`{a} la} Kozono--Yanagisawa \cite{ky}. In addition, our proof uses a polarisation trick that exploits the symmetry of boundary integrals. We learned this trick from an earlier work \cite{c} by Chern. 

\subsection{Boundary value Sobolev spaces for differential forms on manifolds-with-boundary}
Let us first introduce the relevant function spaces. Our notations largely follow Schwarz \cite{s}.

Throughout this work, $(\M,g_0)$ is a Riemannian manifold-with-boundary of dimension $n$. Denote by $\Omega^k(\M)$ the space of differential $k$-forms on $\M$. One can define in standard ways the Sobolev spaces of differential forms; see Hebey \cite{h}. We write $$W^{\ell,p}\Omega^k(\M) \equiv L_\ell^p\Omega^k(\M)$$ for the space of $k$-forms on $\M$ with $W^{\ell,p}$-regularity in the interior. 

We write $d:\Omega^k(\M) \to \Omega^{k+1}(\M)$ for the exterior differential operator, and $\delta: \Omega^{k+1}(\M)\to\Omega^k(\M)$ for the codifferential (\emph{i.e.}, the formal $L^2$-adjoint of $d$). Throughout, the $L^2$- and any other Sobolev norms on $\M$ are all taken with respect to the underlying Riemannian metric $g_0$. The Hodge star $\star: \Omega^k(\M) \to \Omega^{n-k}(\M)$ is also defined via the Riemannian volume measure.

On the space $\mathring{W}^{1,r}\Omega^k(\M):=\left\{u \in L^r\Omega^k(\M):\,du \in L^r\Omega^{k+1}(\M),\, \delta u \in L^r\Omega^{k-1}(\M)\right\}$, we can define the Dirichlet and Neumann trace operators $\ttt$ and $\nnn$:
\begin{align*}
\ttt,\,\nnn: \mathring{W}^{1,r}\Omega^k(\M)\longrightarrow \left[W^{1-\frac{1}{r'},r'}\Omega^k(\p\M)\right]^*,
\end{align*}
where $1/r+1/r'=1$ (the same below). On $C^1\Omega^k(\M)$ the trace operators are given as follows:
\begin{align}\label{t and n}
\ttt \omega := \iota^\# \omega,\qquad \nnn \omega := \omega - \ttt\omega,
\end{align}
where $\iota: \p\M\emb\M$ is the natural inclusion, and $\iota^\#$ is the pullback operator under $\iota$. In other words, for any vectorfields $X_1,\ldots,X_k$ along the boundary $\p\M$ we have
\begin{align*}
\ttt \omega(X_1,\ldots,X_k) \equiv \omega \left(X_1^\top,\ldots,X_k^\top\right),
\end{align*}
where $X^\top$ is the component of $X$ tangential to $\p\M$. 

The notation $\mathring{W}^{1,r}\Omega^k(\M)$ is only temporary; we shall not use it again in the sequel. The point is that, in general (without suitable boundary conditions; and even if so, before proving the Gaffney's inequality) we do not know whether $\mathring{W}^{1,r}\Omega^k(\M)$ coincides with $W^{1,r}\Omega^k(\M)$.  
 
We can now define the boundary value Sobolev spaces using $\ttt$ and $\nnn$:
\begin{eqnarray}
&&W^{1,r}_N\Omega^k(\M) := \left\{ u \in L^r\Omega^k(\M):\,du \in L^r\Omega^{k+1}(\M),\, \delta u \in L^r\Omega^{k-1}(\M),\,\nnn u = 0 \right\},\\
&&W^{1,r}_D \Omega^k(\M) := \left\{ u \in L^r\Omega^k(\M):\,du \in L^r\Omega^{k+1}(\M),\, \delta u \in L^r\Omega^{k-1}(\M),\,\ttt u = 0 \right\}.
\end{eqnarray}
Here ``$D$'' for Dirichlet and ``$N$'' for Neumann.

In addition, let $\nu \in \G(\iota^\# T\M)$ denote the outward unit normal vectorfield on $\p\M$, and let $\nu^\flat$ be the $1$-form canonically dual to $\nu$. By a slight abuse of notations, we still write $\nu$ and $\nu^\flat$ for any of their smooth extensions into the interior of $\M$. Then we set
\begin{eqnarray}
&&\widehat{W^{1,r}_N}\Omega^k(\M) := \left\{ u \in W^{1,r}\Omega^k(\M):\,\nu \mres u |\p\M= 0 \right\},\\
&&\widehat{W^{1,r}_D} \Omega^k(\M) := \left\{ u \in W^{1,r}\Omega^k(\M):\,u \wedge \nu^\flat|\p\M = 0 \right\}.
\end{eqnarray}
Moreover, we put
\begin{eqnarray}
&&C^\infty_N\Omega^k(\M) := \left\{ u \in C^\infty\Omega^k(\clM):\,\nu \mres u|\p\M = 0 \right\},\\
&&C^\infty_D\Omega^k(\M) := \left\{ u \in C^\infty\Omega^k(\clM):\,u\wedge \nu^\flat |\p\M = 0 \right\}.
\end{eqnarray}
Here and hereafter, $\clM :=\M \cup \p\M$ and $\mres$ is the interior multiplication operator. For $\M = $ 3D Euclidean domain and $u =$ vectorfield, $\nu\mres u=0$ and $u \wedge \nu^\flat=0$ are respectively equivalent to $\nu \cdot u=0$ and $u \times \nu =0$ on $\p\M$.

Throughout this note, the $W^{\ell,p}$-norm of a differential form $u$ will be denoted by $\|u\|_{W^{\ell,p}(\M)}$. We suppress the index $k$, the order of the differential form, for notational convenience. Also, for two $L^2$-differential $k$-forms $\alpha$ and $\beta$ over $\M$, we write
\begin{align*}
(\alpha,\beta)_{g_0} := \int_\M \alpha \wedge \star\beta.
\end{align*}
This is the $L^2$-inner product with respect to the metric $g$. The right-hand side is an integral of $n$-forms over an $n$-dimensional manifold, which is defined in the usual manner. We also write
\begin{align*}
\langle \alpha,\beta\rangle_{\bigwedge^k} := \alpha \wedge \star\beta
\end{align*}
to emphasise the paring of $k$-forms.

For further developments we also need the following notations. First, $\na\equiv \na_{g_0}$ denotes the covariant derivative induced by the Levi-Civita connection of the given metric $g_0$. It extends to differential forms of any order via the Leibniz rule. Second, $\na^\dagger$ denotes the formal $L^2$-adjoint of $\na$, again with respect to the Riemannian volume measure of $g_0$. Third, $$\dd\Sigma=\nu \mres \dvg$$ is the Riemannian volume (surface) form on $\p\M$.

For background materials on global/geometric analysis, we refer to Petersen \cite{p}, Hebey \cite{h}, and the first chapter of Schwarz \cite{s}.

\subsection{Main Theorem}
Our main result is the following:
\begin{theorem*}[A]
Let $(\M,g_0)$ be an $n$-dimensional compact Riemannian manifold-with-boundary, let $k \in \{0,1,2,\ldots,n\}$, and let $r \in ]1,\infty[$ be arbitrary. Consider $u \in \widehat{W^{1,r}_D}\Omega^k(\M)$  or $\widehat{W^{1,r}_N}\Omega^k(\M)$, a differential $k$-form subject to the Dirichlet or Neumann boundary condition. Then
\begin{equation*}
\|u\|_{W^{1,r}(\M)} \leq C\left\{\|du\|_{L^r(\M)} + \|\delta u\|_{L^r(\M)} + \|u\|_{L^r(\M)}\right\},
\end{equation*}
where $C$ depends only on $r$, $k$, $n$, and the geometry of $\M$.
\end{theorem*}

\subsection{Variational approach}
To prove Theorem~A, we shall use a variational characterisation of the $W^{1,r}$-norm of differential forms in  boundary value spaces.

\begin{theorem*}[B]
Let $(\M,g_0)$ be an $n$-dimensional compact Riemannian manifold-with-boundary, let $C^\infty_\diamondsuit\Omega^k(\M)$ denote either $C^\infty_{D}\Omega^k(\M)$ or $C^\infty_{N}\Omega^k(\M)$, let $r$ and $q$ be any numbers in $]1,\infty[$, and let $k\in\{0,1,2,\ldots,n\}$. Assume that $u \in \widehat{W^{1,q}_\diamondsuit} \Omega^k(\M,g_0)$ and that the quantity below is finite:
\begin{equation*}
\sup\left\{\frac{\left| ( d u, d \phi)_{g_0} +( \delta u, \delta \phi)_{g_0} +( u,\phi)_{g_0} \right|}{\|d\phi\|_{L^{r'}(\M,g_0)} + \|\delta\phi\|_{L^{r'}(\M,g_0)}+\|\phi\|_{L^{r'}(\M,g_0)}}:\,\phi \in C^\infty_\diamondsuit\Omega^k(\M) \right\}.
\end{equation*}
Then $u \in\widehat{W^{1,r}_\diamondsuit} \Omega^k(\M,g_0)$.

In addition, the following variational inequality is satisfied:
\begin{equation*}
\|u\|_{W^{1,r}} \leq C\sup \left\{\frac{\left| ( d u, d \phi)_{g_0} +( \delta u, \delta \phi)_{g_0} + (u,\phi)_{g_0} \right|}{\|d\phi\|_{L^{r'}(\M,g_0)} + \|\delta\phi\|_{L^{r'}(\M,g_0)}+\|\phi\|_{L^{r'}(\M,g_0)}}:\,\phi \in C^\infty_\diamondsuit\Omega^k(\M) \right\},
\end{equation*}
where $C$ depends only on $r$, $n$, $k$, and the $C^1$-geometry of $(\M,g_0)$.
\end{theorem*}

The case for $n=3$, $g_0$ = the Euclidean metric, and $k=1$ is due to Kozono--Yanagisawa \cite{ky}, based on which a new, refined proof of the Hodge decomposition theorem for vectorfields is obtained. We shall essentially follow the strategies in \cite{ky} to prove Theorem~B.

\subsection{Organisation}
The remaining parts of this note is organised as follows.

First, in \S\ref{sec: prelim} we collect a few preliminary results on geometry. Next, in \S\ref{sec: var}, a proof of Theorem~B shall be given by generalising the arguments in \cite{ky} for vectorfieldson 3D domains. This shall be utilised in \S~\ref{sec: gaffney} to prove Theorem~A. Indeed, our arguments in \S\ref{sec: gaffney} are global. Finally, several concluding remarks will be given in \S\ref{sec: rem}.

\section{Preliminaries}\label{sec: prelim}

Let us first note the following commutation relations between the differential, the codifferential, the Hodge star, and the trace operators:
\begin{lemma}\label{lem: n and t}
For any sufficiently regular differential $k$-form $\omega$ and $(k+1)$-form $\eta$ on the manifold-with-boundary $(\M^n,g_0)$, we have
\begin{equation*}
\begin{cases}
\star (\nnn \omega) = \ttt(\star \omega),\quad
\star(\ttt\omega) = \nnn(\star\omega);\\
\ttt(d\omega) = d(\ttt\omega),\quad \nnn(\delta\omega) = \delta(\nnn \omega).
\end{cases}
\end{equation*}
In addition, it holds that
\begin{align*}
\ttt \omega \wedge \star \nnn \eta = \langle \omega, \nu \mres \eta \rangle_{\Lambda^k}  \,\dd \Sigma,
\end{align*}
where $\dd\Sigma$ is the Riemannian volume measure on $\p\M$ induced from $g_0$. 
\end{lemma}

\begin{proof}
See Proposition~1.2.6, p.27 in Schwarz \cite{s}.   \end{proof}

The following integration by parts formula can be deduced directly from the Stokes'/Gauss--Green's theorem:
\begin{lemma}
Let $\omega$ and $\eta$ be as in Lemma~\ref{lem: n and t}. Then
\begin{align*}
(d\omega,\eta)_{g_0} = (\omega,\delta\eta)_{g_0} + \int_{\p\M} \ttt\omega\wedge\star\nnn\eta.
\end{align*}
\end{lemma}
\begin{proof}
See Proposition~2.1.2, p.60 in Schwarz \cite{s}.   \end{proof}

In addition, we have the Sobolev embedding theorems (\cite{h}) for differential forms on compact $(\M^n,g_0)$: $W^{1,q}\Omega^k(\M) \emb L^r\Omega^k(\M)$ for $1\leq r< \frac{nq}{n-q}$, and $W^{1,s}\Omega^k(\M) \emb C^0\Omega^k(\M)$ for $s>n$. 

\section{Variational Inequalities}\label{sec: var}

To prove Theorem~B, we first argue for $\M =\R^n$ with $u$ compactly supported. Then we deduce the assertion for $\M = \overline{\R^n_+}$ by even or odd extensions across the boundary. The general case holds via a partition of unity argument, with the covering charts being sufficiently refined. 


\subsection{The case of $\R^n$}\label{subsec a1}
In this subsection $(\bullet,\bullet)$ always denotes the standard Euclidean inner product. We shall prove
\begin{lemma}\label{lem: whole space}
Let $u \in W^{1,q}\Omega^k(\R^n)$ for $q \in ]1,\infty[$. Assume that there is an $r \in ]1,\infty[$ such that
\begin{align*}
\sup\left\{\frac{\left| (d u, d \phi) +( \delta u, \delta \phi) + ( u,\phi)\right|}{\|d\phi\|_{L^{r'}(\R^n)} + \|\delta\phi\|_{L^{r'}(\R^n)}+\|\phi\|_{L^{r'}(\R^n)}}  :\,\phi \in C^\infty_{c}\Omega^k(\R^n) \right\} < \infty.
\end{align*} 
Then $u \in W^{1,r}\Omega^k(\R^n)$ with the estimate
\begin{equation}\label{r-est on Rn, lemma}
\|u\|_{W^{1,r}(\R^n)} \leq C\sup\left\{\frac{\left| ( d u, d \phi) +( \delta u, \delta \phi) + ( u,\phi) \right|}{\|d\phi\|_{L^{r'}(\R^n)} + \|\delta\phi\|_{L^{r'}(\R^n)}+\|\phi\|_{L^{r'}(\R^n)}}  :\,\phi \in C^\infty_{c}\Omega^k(\R^n) \right\},
\end{equation}
where $C$ depends only on $r$, $k$, and $n$.
\end{lemma}

\begin{proof}
Consider the space
\begin{equation}
\fs_k := \left\{(\1-\Delta)\psi:\,\psi \in C^\infty_{c}\Omega^k(\R^n)\right\},
\end{equation}
where $\1$ is the identity map and $\Delta$ is the Euclidean Laplacian. We \emph{claim} that $\fs_k$ is dense in $L^{r'}\Omega^k(\R^n)$ for any $r \in ]1,\infty[$. 

This statement is proved below by contradiction. Suppose that there were an $f \in L^{r}\Omega^k(\R^n) \sim \{0\}$ annihilating $\fs_k$; that is,
\begin{align}\label{Fourier}
\int_{\R^n} f \bullet (\1-\Delta)\psi = 0 \qquad \text{ for all } \psi \in \fs_k.
\end{align}
The symbol $\bullet$ denotes the usual Euclidean inner product for differential $k$-forms on $\R^n$. If $k=0$, \emph{i.e.}, $f$ and $\psi$ are scalar functions, it then follows that
\begin{align*}
(\1-\Delta)f=0 \qquad \text{ as Schwartz distributions}.
\end{align*}
The Fourier transform $\hat{f}$ of $f$ satisfies $(1+4\pi^2|\xi|^2) \hat{f}(\xi)=0$ for all $\xi \in \R^n$, thus $\hat{f} \equiv 0$ and hence $f \equiv 0$. This yields a contradiction.

For $k \geq 1$, let us expand differential $k$-forms with respect to the basis consisting of simple vectorfields formed by the canonical basis for $\R^n$. That is, consider 
\begin{align*}
f = \sum f_{j_1j_2\ldots j_k} e_{j_1} \wedge \cdots \wedge e_{j_k}
\end{align*}
where $\{e_1,\ldots,e_n\}$ is the Euclidean canonical basis, and the summation is taken over ascending $k$-tuples $(j_1,\ldots,j_k)$ such that $1 \leq j_1 <j_2<\ldots<j_{k} \leq n$; similarly for $\psi$. Then Eq.~\eqref{Fourier} is equivalent to $\int_{\R^n} f_{j_1j_2\ldots j_k}(\1-\Delta)\psi_{j_1j_2\ldots j_k}\,\dd x = 0$
for all such indices. The arguments above implies that $ f_{j_1j_2\ldots j_k} \equiv 0$. Hence $f \equiv 0$, which gives the contradiction and proves the  \emph{claim}.

Now we are ready to prove the lemma. Fix $r \in ]1,\infty[$ and take a test form $\phi = \na_j \psi$ for some $\psi \in C^\infty_{c}\Omega^k(\R^n)$ and $j \in \{1,2,\ldots,n\}$ fixed. Then
\begin{align}\label{new 1}
&\sup\left\{\frac{\left| (d u, d \phi) +(\delta u, \delta \phi)+ (u,\phi) \right|}{\|d\phi\|_{L^{r'}(\R^n)} + \|\delta\phi\|_{L^{r'}(\R^n)}+\|\phi\|_{L^{r'}(\R^n)}}  :\,\phi \in C^\infty_{c}\Omega^k(\R^n) \right\} \nonumber\\
&\quad \geq 
\sup\left\{\frac{\left| (d u, d \na_j\psi) +( \delta u, \delta \na_j\psi) + ( u,\na_j\psi) \right|}{\|d\na_i\psi\|_{L^{r'}(\R^n)} + \|\delta\na_j\psi\|_{L^{r'}(\R^n)}+\|\na_j\psi\|_{L^{r'}(\R^n)}}  :\,\psi \in C^\infty_{c}\Omega^k(\R^n) \right\}.
\end{align}
On $\R^n$ we can commute $d$ with $\na_j$ and $\delta$ with $\na_j$. Moreover, note that $-\Delta=d\delta+\delta d$, so
\begin{align*}
( d u, d \na_j\psi)+( \delta u, \delta \na_j\psi) + ( u,\na_j\psi) = ( \left\{\1-\Delta\right\}\psi, \na_j u ).
\end{align*}

By the standard Calder\'{o}n--Zygmund estimate on $\R^n$
\begin{align*}
\|\na \na \psi\|_{L^{r'}(\R^n)} \leq C\left( \|\Delta\psi\|_{L^{r'}(\R^n)} + \|\psi\|_{L^{r'}(\R^n)} \right)
\end{align*}
and the interpolation inequality
\begin{align*}
\|\na \psi\|_{L^{r'}(\R^n)} \leq C\left(1+\|\na\na\psi\|_{L^{r'}(\R^n)}\right),
\end{align*}
one may continue the previous estimate by
\begin{align*}
&\sup\left\{\frac{\left| (d u, d \phi) +( \delta u, \delta \phi) + ( u,\phi)\right|}{\|d\phi\|_{L^{r'}(\R^n)} + \|\delta\phi\|_{L^{r'}(\R^n)}+\|\phi\|_{L^{r'}(\R^n)}}  :\,\phi \in C^\infty_{c}\Omega^k(\R^n) \right\} \\
&\quad \geq c
\sup\left\{\frac{\left| (\left\{\1-\Delta\right\}\psi, \na_j u )\right|}{ \|\{1-\Delta\}\psi\|_{L^{r'}(\R^n)}}  :\,\psi \in C^\infty_{c}\Omega^k(\R^n) \right\}\\
&\quad = c \sup\left\{ \frac{\left|(\rho, \na_j u )\right|}{ \|\rho\|_{L^{r'}(\R^n)}}  :\,\rho \in \fs_k \right\}.
\end{align*}

By the \emph{claim} above (namely, the density of $\fs_k$ in $L^{r'}\Omega^k(\R^n)$) and the duality characterisation of the $L^r$-norm, it holds that
\begin{align*}
\sup\left\{ \frac{\left| (\rho, \na_j u)\right|}{ \|\rho\|_{L^{r'}(\R^n)}}  :\,\rho \in \fs_k \right\} = \sup\left\{ \frac{\left| ( \rho, \na_j u )\right|}{ \|\rho\|_{L^{r'}(\R^n)}}  :\,\rho \in L^{r'}\Omega^k(\R^n)  \right\}= \|\na_ju\|_{L^r(\R^n)}.
\end{align*}
As the index $j$ is arbitrary, we can chain together the previous inequalities to get
\begin{align*}
\|\na u\|_{L^{r}(\R^n)} \leq C \sup\left\{\frac{\left| ( d u, d \phi)+( \delta u, \delta \phi) +( u,\phi) \right|}{\|d\phi\|_{L^{r'}(\R^n)} + \|\delta\phi\|_{L^{r'}(\R^n)}+\|\phi\|_{L^{r'}(\R^n)}}  :\,\phi \in C^\infty_{c}\Omega^k(\R^n) \right\},
\end{align*}
where $\na$ is the covariant derivative from  differential $k$-forms to $(k+1)$-forms induced by the Levi-Civita connection on $\R^n$. 

Repeating the same arguments with $\phi = \psi$ instead of $\phi = \na_j\psi$, we obtain
\begin{align*}
\|u\|_{L^{r}(\R^n)} \leq C\sup\left\{\frac{\left| ( d u, d \phi)+( \delta u, \delta \phi) + ( u,\phi) \right|}{\|d\phi\|_{L^{r'}(\R^n)} + \|\delta\phi\|_{L^{r'}(\R^n)}+\|\phi\|_{L^{r'}(\R^n)}}  :\,\phi \in C^\infty_{c}\Omega^k(\R^n) \right\}.
\end{align*}
The previous two estimates lead to Eq.~\eqref{r-est on Rn, lemma}. All the constants $C$ and $c$ in this proof depend on nothing but $r$, $n$, and $k$. Therefore, $u \in W^{1,r}\Omega^k(\R^n)$ and the proof is complete.   \end{proof}

\begin{remark}
Since Lemma~\ref{lem: whole space} is only concerned with  differential forms on $\R^n$ and our arguments utilise the Fourier transform, the Laplacian we considered in the proof is the \emph{Hodge Laplacian}, which differs from the Laplace--Beltrami operator by a sign. In contrast, when working with manifolds in the later parts of the paper, for notational convenience we shall always take $\Delta$ to be the Laplace--Beltrami operator, \emph{i.e.}, $\Delta=d\delta+\delta d$.
\end{remark}

\subsection{The case of $\overline{\R^n_+}$} \label{subsec a2}

The analogue of Lemma~\ref{lem: whole space} holds for the halfspace $\overline{\R^n_+}$. As in the previous subsection, 
$(\bullet,\bullet)$ is reserved for the standard Euclidean inner product in Lemma~\ref{lem: halfspace} below. Let $\widehat{W^{1,q}_\diamondsuit} \Omega^k(\M,g_0)$ denote either $\widehat{W^{1,q}_N} \Omega^k(\M,g_0)$ or $\widehat{W^{1,q}_D} \Omega^k(\M,g_0)$.

\begin{lemma}\label{lem: halfspace}
Let $u \in \widehat{W^{1,q}_\diamondsuit} \Omega^k(\half)$ for some $q \in ]1,\infty[$. Assume that 
\begin{align}\label{assumption, Rn+}
\sup\left\{\frac{\left| (d u, d \phi) +( \delta u, \delta \phi) + ( u,\phi) \right|}{\|d\phi\|_{L^{r'}(\R^n_+)} + \|\delta\phi\|_{L^{r'}(\R^n_+)}+\|\phi\|_{L^{r'}(\R^n_+)}}  :\,\phi \in C^\infty_{\diamondsuit} \Omega^k(\R^n_+) \right\} < \infty
\end{align} 
for some $r \in ]1,\infty[$. Then $u \in W^{1,r}\Omega^k(\R^n_+)$ with the estimate
\begin{equation}\label{r-est on Rn+, lemma}
\|u\|_{W^{1,r}(\R^n_+)} \leq C\sup\left\{\frac{\left| ( d u, d \phi) +( \delta u, \delta \phi) + (u,\phi)\right|}{\|d\phi\|_{L^{r'}(\R^n_+)} + \|\delta\phi\|_{L^{r'}(\R^n_+)}+\|\phi\|_{L^{r'}(\R^n_+)}}  :\,\phi \in C^\infty_{\diamondsuit} \Omega^k(\R^n_+) \right\},
\end{equation}
where the constant $C$ depends only on $r$, $n$, and $k$.
\end{lemma}

\begin{proof}
We shall prove for $\diamondsuit = D$ and $\diamondsuit = N$ separately. 

\noindent
\underline{Case 1: $u \in \widehat{W^{1,q}_N} \Omega^k(\half)$.} We consider the extension:
\begin{equation*}
\widetilde{u}(x) := \begin{cases}
u(x)\qquad \text{ if } x^n \geq 0,\\
 \ttt u (x^\star) \oplus -\nnn u (x^\star) 
\qquad \text{ if } x^n < 0,
\end{cases}
\end{equation*}
where $x=(x^1,\ldots,x^n)^\top$ and $x^\star := (x^1,\ldots,x^{n-1},-x^n)^\top$. 

By assumption, Eq.~\eqref{assumption, Rn+} holds for $u$. It also holds for $\widetilde{u}$ (which lies in $W^{1,q}\Omega^k(\half)$) with respect to test $k$-forms $\phi \in C^\infty_{c}\Omega^k(\R^n)$. To see this, note first that the boundary condition for $u$ ensures the (weak) differentiability of $\widetilde{u}$ across $\p\half$. Thus we have $\int_{\R^n}|\na\widetilde{u}|^r = \int_\half |\na\widetilde{u}|^r + \int_{\R^n_-}|\na\widetilde{u}|^r$, so by construction
\begin{align*}
\|\na \widetilde{u}\|_{L^r(\R^n)}^r = 2 \|\na u\|_{L^r(\half)}^r.
\end{align*}
Similarly we get 
\begin{align*}
\| \widetilde{u}\|_{L^r(\R^n)}^r = 2 \| u\|_{L^r(\half)}^r.
\end{align*}
It thus follows from Lemma~\ref{lem: whole space} that $\widetilde{u} \in W^{1,r}\Omega^k(\R^n)$ and the following estimate holds:
\begin{equation}\label{r-est on Rn+, lemma, variant}
\|\widetilde{u}\|_{W^{1,r}(\half)} \leq C\sup\left\{\frac{\left|( d \widetilde{u}, d \phi)_{\R^n} +( \delta \widetilde{u}, \delta \phi)_{\R^n} + ( \widetilde{u},\phi)_{\R^n} \right|}{\|d\phi\|_{L^{r'}(\R^n)} + \|\delta\phi\|_{L^{r'}(\R^n)}+\|\phi\|_{L^{r'}(\R^n)}}  :\,\phi \in C^\infty_{c}\Omega^k(\R^n) \right\}.
\end{equation}

The passage from Eq.~\eqref{r-est on Rn+, lemma, variant} to Eq.~\eqref{r-est on Rn+, lemma} relies on choosing special test $k$-forms $\phi$. Indeed, for any $\phi \in C^\infty_{c}\Omega^k(\R^n)$ we set
\begin{align*}
\underline{\phi} (x) := \left[ \ttt\phi(x) + \ttt\phi(x^\star) \right] \oplus \left[ \nnn\phi(x) - \nnn\phi(x^\star) \right].
\end{align*}
Thus, the restriction of $\underline{\phi}$ to $\half$ lies in $C^\infty_N\Omega^k(\half)$. Moreover, for $T \in \{d,\delta,{\bf id}\}$ one has
\begin{align}\label{new2}
( T\widetilde{u},\phi)_{\R^n} = (Tu, \underline{\phi})_{\R^n_+}.
\end{align}
Therefore,
\begin{align*}
\|u\|_{W^{1,r}(\half)} &\leq C\|\widetilde{u}\|_{W^{1,r}(\R^n)}\\
&\leq C\sup\left\{\frac{\left| (d \widetilde{u}, d \phi)_{\R^n} +( \delta \widetilde{u}, \delta \phi)_{\R^n} +( \widetilde{u},\phi)_{\R^n} \right|}{\|d\phi\|_{L^{r'}(\R^n)} + \|\delta\phi\|_{L^{r'}(\R^n)}+\|\phi\|_{L^{r'}(\R^n)}}  :\,\phi \in C^\infty_{c}\Omega^k(\R^n) \right\}\\
&\leq C\sup\left\{\frac{\left| \left(d {u}, d \overline{\phi}\right)_{\half} +\left(\delta {u}, \delta \overline{\phi}\right)_{\half} + \left({u},\overline{\phi}\right)_{\half} \right|}{\left\|d\overline{\phi}\right\|_{L^{r'}(\half)} + \left\|\delta\overline{\phi}\right\|_{L^{r'}(\half)}+\left\|\overline{\phi}\right\|_{L^{r'}(\half)}}  :\,\underline{\phi} \in C^\infty_N\Omega^k(\half) \right\}.
\end{align*}
All the constants in the above arguments depend only on $r$, $n$, and $k$. This completes the proof for  $u \in \widehat{W^{1,q}_N} \Omega^k(\half)$.

\smallskip
\noindent
\underline{Case 2: $u \in \widehat{W^{1,q}_D} \Omega^k(\half)$.} In this case, consider the extension
\begin{equation*}
\widetilde{\widetilde{u}}(x):=\begin{cases}
u(x)\qquad \text{ if } x^n \geq 0,\\
-\ttt u(x^\star) \oplus \nnn u(x^\star)\qquad \text{ if } x^n<0.
\end{cases}
\end{equation*}
Correspondingly, for any test differential form $\phi\in C^\infty_{c}\Omega^k(\R^n)$ we set 
\begin{align*}
\underline{\underline{\phi}} (x) := \left[ \ttt\phi(x) - \ttt\phi(x^\star) \right] \oplus \left[ \nnn\phi(x) + \nnn\phi(x^\star) \right].
\end{align*}
Then the restriction of $\underline{\underline{\phi}}$ to $\half$ is an element of $\widehat{W^{1,q}_D} \Omega^k(\half)$, and all the previous arguments for $u \in\widehat{W^{1,q}_N} \Omega^k(\half)$ carry over the case $u \in \widehat{W^{1,q}_D} \Omega^k(\half)$.  \end{proof}

In what follows, we denote by $\eu$  the Euclidean metric, and by $(\bullet,\bullet)_\eu$ the Euclidean inner product on either $\R^n$ or $\half$, which shall be clear from the context.

Lemma~\ref{lem: halfspace} also holds if $\half$ is equipped with any constant metric (possibly different from the Euclidean metric):

\begin{corollary}\label{cor: halfspace, constant metric}
Assume that $\gc$ is a constant metric on $\Phi(U)\subset \half$. Then Lemma~\ref{lem: halfspace} continues to hold, with all the inner products and  Sobolev norms therein taken with respect to $\gc$, and with the constant $C$ depending only on $r$, $k$, $n$, and $\gc$.
\end{corollary}

\begin{proof}
Consider the quotient
\begin{align*}
\frac{\left|( d u, d \phi)_\gc +(\delta u, \delta \phi)_\gc + (u,\phi)_\gc \right|}{\|d\phi\|_{L^{r'}(\R^n_+,\gc)} + \|\delta\phi\|_{L^{r'}(\R^n_+,\gc)}+\|\phi\|_{L^{r'}(\R^n_+,\gc)}}
\end{align*}
for any test $k$-form $\phi \in C^\infty_{c}\Omega^k(\R^n_+)$. We \emph{claim} that it differs from  its Euclidean analogue
\begin{align*}
\frac{\left| (d u, d \phi)_\eu +(\delta u, \delta \phi)_\eu  + (u,\phi)_\eu  \right|}{\|d\phi\|_{L^{r'}(\R^n_+)} + \|\delta\phi\|_{L^{r'}(\R^n_+)}+\|\phi\|_{L^{r'}(\R^n_+)}}
\end{align*}
only by a constant depending on $\gc$, $r$, and $k$.

Indeed, for arbitrary differential $\ell$-forms $\alpha$ and $\beta$ supported on $\Phi(U)\subset \half$, we can express in local co-ordinate frames that 
\begin{align*}
(\alpha,\beta)_{\gc} &= \int_{\Phi(U)} \gc^{i_1j_1}\cdots\gc^{i_\ell j_\ell} \alpha_{i_1\ldots i_\ell}\beta_{j_1\ldots j_\ell}\sqrt{\det\,\gc} \,\dd x; \\
(\alpha,\beta)_\eu &= \int_{\Phi(U)} \alpha_{i_1\ldots i_\ell}\beta_{j_1\ldots j_\ell}\,\dd x.
\end{align*}
Thus
\begin{align*}
c(\alpha,\beta)_\eu \leq (\alpha,\beta)_\gc \leq C(\alpha,\beta)_\eu
\end{align*}
with $c$ and $C$ depends only on $\ell$, as well as the (pointwise) upper bound for $\gc$ and lower bound for $\left(\gc\right)^{-1}$. Similar computations in local co-ordinates also yield that
\begin{align*}
c'\|\Upsilon\|_{L^{r'}(\half,\eu)} \leq\|\Upsilon\|_{L^{r'}(\half,\gc)} \leq C'\|\Upsilon\|_{L^{r'}(\half,\eu)} 
\end{align*}
for any differential $\ell$-form $\Upsilon$ supported on $\Phi(U)$ with finite $L^{r'}$-norm. Here $c'$, $C'$ depend only on $\gc$, $\ell$, and $r$. Then the \emph{claim} follows once we choose $\ell$, $\alpha$, $\beta$, and $\Upsilon$ appropriately. 

The same argument shows that for a given differential $k$-form, its $W^{1,r}$-norms with respect to $\gc$ and $\eu$ differ only by a constant depending on $\gc$,  $n$, $\ell$, and $r$. The proof is now complete in view of Lemma~\ref{lem: halfspace}.   \end{proof}

\subsection{The case of a small chart} \label{subsec a3}

In this subsection, we show that the conclusion of Lemma~\ref{lem: whole space} continues to hold for differential forms defined over one sufficiently small chart $U_\alpha\subset\overline{\M}=\M \cup \p\M$. By definition, there is a diffeomorphism $\Phi_\alpha: U_\alpha \to \Phi_\alpha(U_\alpha) \subset \overline{\R^n_+}$ such that image $\Phi_\alpha(U_\alpha)$ is relatively open in the closed half space. For obvious reasons, we shall focus only on \emph{boundary} charts, \emph{i.e.}, those $U_\alpha$ with $U_\alpha \cap \p\M \neq \emptyset$. In the rest of this subsection, we work with one fixed $U_\alpha$, hence the subscripts ${}_\alpha$ shall be systematically dropped.

We first prepare ourselves with several simple geometric estimates. Let $\kappa>0$ be arbitrary.  On a sufficiently small chart $U$, the metric $g_0|U$ is almost constant. More precisely, for an arbitrary reference point $P \in U$, one has 
\begin{equation}\label{metric, 0}
\left\|g_0 - g_0(P)\right\|_{C^0(U)} + \left\|Dg_0 - [Dg_0](P)\right\|_{C^0(U)}  \leq \kappa.
\end{equation}
Here the derivative $Dg_0$ is  understood as the $3$-tensor field $[Dg_0]_{ijk}=\p_i (g_0)_{jk}$; the metric components $(g_0)_{jk}$ are taken with respect some (\emph{a priori} given) local co-ordinate system on $U$.


Choose the constant metric
\begin{align*}
\gc := \Phi_\#\left( g_0(P) \right)
\end{align*}
on $\Phi(U)$; it is indeed a  Riemannian metric since $\Phi$ is a diffeomorphism. It then follows from Eq.~\eqref{metric, 0} that
\begin{equation}\label{metric est}
\left\|\Phi_\# g_0 - \gc\right\|_{C^0(U)} + \left\|D[\Phi_\# g_0] - \Phi_\#[Dg_0(P)]\right\|_{C^0(U)} \leq C_1\kappa,
\end{equation}
where $C_1$ depends only on the $C^1$-geometry of the chart $U$. That is, the pushforward metric $\Phi_\#g_0$ is $\ok$-close in the $C^1$-topology to the constant metric $\gc$.

Next, note that the $C^0$-bound in \eqref{metric est} and the continuity of determinant in the uniform topology give us 
\begin{equation}\label{det est}
\begin{cases}
\left\|\det g_0 - \det \left(g_0(P)\right)\right\|_{C^0(U)} \leq C_2\kappa,\\
\left\|\det\left(\Phi_\# g_0\right) - \det \gc\right\|_{C^0(\Phi(U))} \leq C_3\kappa,
\end{cases}
\end{equation}
with $C_2$ and $C_3$ depending only on the $C^1$-geometry of the chart $U$. 

To compare inverse metrics $\Phi_\#\left(g_0^{-1}\right)$ and $\gc^{-1}$, recall for any metric $g$ the Cramer's rule:\begin{align*}
g^{-1} = \frac{{\rm Adj}\,g}{\det\,g},
\end{align*}
where ${\rm Adj}\,g$ is the adjugate matrix of $g$. Modulo natural duality isomorphisms, it holds that
\begin{align*}
{\rm Adj}\,g \simeq\bigwedge^{n-1} g,
\end{align*}
the $(n-1)$-fold wedge product of $g$ as a matrix.  Hence
\begin{align*}
\Phi_\# \left(g_0^{-1}\right) - \gc\,^{-1} = \frac{\left\{\left( \bigwedge^{n-1} \Phi_\#g_0\right)\det \gc - \left(\bigwedge^{n-1}\gc\right) \det\left[ \Phi_\#(g_0) \right] \right\}}{\det\left[ \gc \cdot \Phi_\#g_0 \right]}.
\end{align*}
For the denominator we can bound
\begin{align*}
\det\left[ \gc \cdot \Phi_\#g_0 \right] \geq \left(\det\,\gc\right)^2 - C_4\kappa > 0
\end{align*}
thanks to Eq.~\eqref{det est}. Here $C_4$ depends on the $C^1$-geometry of $U$. To control the numerator, we apply the simple  combinatorial identity
\begin{align*}
\left(\prod_{i=1}^n a_i\right) - \left(\prod_{i=1}^n b_i\right) = \sum_{i=1}^n \left\{ \left(\prod_{1\leq k \leq i}b_k\right)\left(a_i-b_i\right)\left(\prod_{i\leq  j \leq n} a_j\right) \right\}
\end{align*} 
together with Eqs.~\eqref{det est} and \eqref{metric est}. In this way we get
\begin{equation}\label{inverse metric est}
\left\|
\Phi_\# \left(g_0^{-1}\right) - \gc\,^{-1}\right\|_{C^0(\Phi(U))} \leq C_5\kappa.
\end{equation}
Here $C_5$ depends again only on the $C^1$-geometry of $U$.

Furthermore, combining Eqs.~\eqref{inverse metric est} and \eqref{metric est} allows us to compare the Christoffel symbols $\G^i_{jk}$ and $\overline{\G^i_{jk}}$ for $\Phi_\#\left(g_0\right)$ and $\gc$, respectively:
\begin{equation}\label{christoffel est}
\left\|\G^i_{jk} -\overline{\G^i_{jk}}\right\|_{C^0(\Phi(U))} \leq C_6\kappa.
\end{equation}

To summarise, under the assumption that $g_0$ is $\mathcal{O}(\kappa)$-close to a constant metric in $C^1$-topology (Eq.~\eqref{metric, 0}), we have proved that $g_0$ is, in fact, $\mathcal{O}(\kappa)$-close to the constant metric in the bi-$C^1$-topology. The constant in $\mathcal{O}(\kappa)$ can be chosen to depend only on the $C^1$-topology of $U$. The same conclusion remains valid under the pushforward via $\Phi$.


\begin{proposition}\label{prop: chart}
Let $(\M,g_0)$ be a Riemannian manifold-with-boundary. For any  sufficiently small boundary chart $U \subset \overline{\M}$ (\emph{i.e.}, $U \cap \p\M\neq\emptyset$),   the following result holds --- 

Let $u \in W^{1,q}_{\diamondsuit}\Omega^k(\M)$ be a differential $k$-form compactly supported in $U$, where $q \in ]1,\infty[$ and $\diamondsuit=D$ or $N$. Assume that for some index $r \in ]1,\infty[$ one has
\begin{align}\label{assumption, chart}
\sup\left\{\frac{\left| ( d u, d \phi)_{g_0} +( \delta u, \delta \phi)_{g_0} + (u,\phi)_{g_0} \right|}{\|d\phi\|_{L^{r'}(U,g_0)} + \|\delta\phi\|_{L^{r'}(U,g_0)}+\|\phi\|_{L^{r'}(U,g_0)}}  :\,\phi \in C^\infty_{\diamondsuit}\Omega^k(U) \right\} < \infty.
\end{align} 
Then $u \in W^{1,r}_\diamondsuit\Omega^k(U,g_0)$ and the following estimate is valid:
\begin{equation}\label{r-est on chart, lemma}
\|u\|_{W^{1,r}(U,g_0)} \leq C\sup\left\{\frac{\left| (d u, d \phi)_{g_0} +( \delta u, \delta \phi)_{g_0} + (u,\phi)_{g_0} \right|}{\|d\phi\|_{L^{r'}(U,g_0)} + \|\delta\phi\|_{L^{r'}(U,g_0)}+\|\phi\|_{L^{r'}(U,g_0)}}  :\,\phi \in C^\infty_\diamondsuit\Omega^k(U) \right\}.
\end{equation}
The constant $C$ depends only on $r$, $k$, and the $C^1$-geometry of the chart $(U,g_0)$. 
\end{proposition}

\begin{remark}\label{rem: C1 geometry}
The sufficiently small chart $U$ can be chosen uniformly in the following sense: Given  an arbitrarily small number $\kappa>0$, there is some number $\rho>0$ depending only on $\kappa$ and the $C^1$-geometry of \emph{the whole manifold} $(\M,g_0)$, such that whenever a boundary chart $U$ has intrinsic diameter smaller than $\rho$, then it is a valid choice for Proposition~\ref{prop: chart}. 

In view of the earlier  arguments in this subsection,  we may choose $\rho$ such that on any $U$ with diameter smaller than $\rho$, one has
\begin{equation*}
\left\|g_0 - g_0(P)\right\|_{C^0(U)} + \left\|Dg_0 - [Dg_0](P)\right\|_{C^0(U)}  \leq \kappa.
\end{equation*}
\end{remark}


The proof of Proposition~\ref{prop: chart} is technical yet straightforward: all we need to do is to compare the inner products of differential forms on $U$ with their pushed-forward versions on $(\half, \gc)$, which can be dealt with by Corollary~\ref{cor: halfspace, constant metric}. 

In what follows, for any differential form $\alpha$ on $U \subset \overline{\M}$ we set
\begin{equation}\label{natural}
\alpha^\natural:=\Phi_\#\alpha.
\end{equation}
Let us first check that the boundary conditions for $W^{1,q}_\diamondsuit \Omega^k(\M)$ are preserved under $\Phi_\#$. This relies only on  differentiable structures of $\M$ and $\half$, but not on Riemannian structures.

\begin{lemma}\label{lem: bdry cond}
Assume that $\phi \in C^\infty_\diamondsuit\Omega^k(\M)$ is supported in a boundary chart $U$. Then $\phi^\natural\in C^\infty_\diamondsuit\Omega^k(\half)$. Conversely, if $\varphi^\natural\in C^\infty_\diamondsuit\Omega^k(\half)$  is supported in the image $\Phi(U)$ of a boundary chart $U$, then $\Phi^\#\varphi^\natural \in C^\infty_\diamondsuit\Omega^k(\M)$.
\end{lemma}

\begin{proof}
Let $\iota: \p\M \emb \overline{\M}$ be the natural inclusion. By definition, for $\phi \in C^\infty_N\Omega^k(\M)$ we have $\ttt\phi=0$ on $\p\M$, namely that $\iota^\# \phi =0$ as a differential $k$-form over $\p\M$. The natural inclusion $\overline{\iota}:\p\half \emb \overline{\half}$ is given by $$\overline{\iota} := \Phi \circ \iota \circ \Phi^{-1}.$$ As $\Phi$ is a diffeomorphism, we can safely pullback and pushforward differential forms via $\Phi$. Thus we have
\begin{align*}
\overline{\iota}^\# \phi^\natural &= \left(\Phi \circ \iota \circ \Phi^{-1}\right)^\# \left(\Phi_\#\phi\right)\\
&= (\Phi^{-1})^\# \circ \iota^\# \circ \Phi^\# \circ \Phi_\#\phi\\
&= \Phi_\# \left(\iota^\#\phi\right) \\
&= 0,
\end{align*}
\emph{i.e.}, $\ttt\phi^\natural=0$ and $\phi^\natural \in C^\infty_N\Omega^k(\half)$.

Now we assume that $\phi \in  C^\infty_D\Omega^k(\M)$. Then $\nnn\phi=0$ on $\p\M$, \emph{i.e.}, $\phi - \iota^\#\phi=0$. The above computation gives us
\begin{align*}
\nnn\phi^\natural = \phi^\natural-\ttt\phi^\natural = \Phi_\#\left( \phi-\iota^\#\phi \right) = 0. 
\end{align*}
Hence $\phi^\natural \in C^\infty_D\Omega^k(\half)$. 

To complete the proof, we simply notice that all the previous arguments can validly run backwards, as $\Phi$ is a diffeomorphism.
\end{proof}

Next, recall that $$\gc=\Phi_\#(g_0);$$ one may easily compare 
$(\alpha,\beta)_{g_0}$ with $\left(\alpha^\natural,\beta^\natural\right)_{\gc}$ via the following
\begin{lemma}\label{lem: inner prod comparison}
Let $\alpha$, $\beta$ be arbitrary differential $\ell$-forms on the chart $(U,\Phi) \subset (\M,g_0)$, and let $\kappa >0$ be arbitrary. There is a constant $C>0$ such that 
\begin{align*}\left|\left(\alpha^\natural,\beta^\natural\right)_{\gc} 
-(\alpha,\beta)_{g_0} \right|\leq C{\kappa}.
\end{align*}
Here $C$ depends only on $\kappa$, $n$, $\ell$, and the $C^1$-geometry of $U$.
\end{lemma}

\begin{proof}
We compute in local co-ordinates. Write $y = \Phi(x)$, $\alpha = \sum \alpha_{i_1\ldots i_\ell}dx^{i_1}\wedge\cdots\wedge dx^{i_\ell}$, and $\beta = \sum \beta_{j_1\ldots j_\ell}dx^{j_1}\wedge\cdots\wedge dx^{j_\ell}$, with summations taken over ascending $\ell$-tuples of indices $1 \leq i_1 <\ldots<i_\ell \leq n$ and $1 \leq j_1 <\ldots< j_\ell \leq n$. Then
\begin{align*}
\langle\alpha,\beta\rangle_{g_0} = \int_U g_0^{i_1 j_1}\cdots g_0^{i_\ell j_\ell} \alpha_{i_1\cdots i_\ell} \beta_{j_1\cdots j_\ell}\,\dd V_{g_0},
\end{align*}
where $\dd V_{g_0}$ is the Riemannian volume form. (Here the summation convention is assumed.) By Eqs.~\eqref{inverse metric est} and \eqref{metric est} we have
\begin{align*}
&\left\| g_0^{-1} - g_0^{-1}(P) \right\|_{C^0(U)} \leq C_7\kappa,\\
&\left\|\dd V_{g_0} - \dd V_{g_0(P)}\right\|_{C_0(U)} \leq C_8\kappa.
\end{align*}
It thus follows that
\begin{align}\label{C9}
&\left|(\alpha,\beta)_{g_0}- ( \alpha,\beta)_{g_0(P)}\right|\nonumber\\
&\quad = \left|\int_U g_0^{i_1 j_1}\cdots g_0^{i_\ell j_\ell} \alpha_{i_1\cdots i_\ell} \beta_{j_1\cdots j_\ell}\,\dd V_{g_0} - \int_{U}g_0(P)^{i_1 j_1}\cdots g_0(P)^{i_\ell j_\ell} \alpha_{i_1\cdots i_\ell} \beta_{j_1\cdots j_\ell}\,\dd V_{g_0(P)}\right|\nonumber\\
&\quad \leq C_9\kappa. 
\end{align}
In addition, in view of the definition of $\natural$ and $\gc$, we know that $\Phi$ is an isometry from $(U,g_0(P))$ to $(\Phi(U),\gc)$. That is, $$(\alpha,\beta)_{g_0(P)}=\left(\alpha^\natural,\beta^\natural\right)_\gc.$$ The constant $C_9$ depends only on $\kappa$, $n$, $\ell$, and the $C^1$-geometry of $U$.    \end{proof}

It is clear that Lemma~\ref{lem: inner prod comparison} remains valid for $\alpha$, $\beta$ with weaker regularities, as long as the pairings $(\alpha,\beta)$ and $\left(\alpha^\natural, \beta^\natural\right)_{\half}$ are well defined in the sense of distributions and $\alpha^\natural$, $\beta^\natural$ are well defined via pushforward. In particular, it holds when $\alpha \in L^r$ and $\beta \in L^{r'}$ for some $r \in ]1,\infty[$.

We next observe that $d$ and $\delta$ commute with $\natural$:
\begin{lemma}\label{lem: compare diff op}
For any differential $k$-forms $\alpha$ and $\beta$ on $U \subset \overline{M}$, we have
\begin{equation*}
(d\alpha)^\natural=d\left(\alpha^\natural\right),\qquad \left(\delta\alpha\right)^\natural = \delta\left(\alpha^\natural\right).
\end{equation*}
\end{lemma}

The former identity relies solely on the differentiable structure, while the latter relies additionally on the Riemannian structure. In general, the latter holds whenever the codifferentials are compatible with pushforward in the following sense: for a diffeomorphism $\Phi: (\M,g)\to(\M',g')$  that is \emph{isometric} and for any differential form $\alpha$ on $\M$, we have
\begin{align*}
\Phi_\#\left[\delta_{(\M,g)} \alpha\right] = \delta_{(\M',g')}\left[\Phi_\#\alpha\right].
\end{align*}

\begin{proof}
The identity for $d$ holds by
\begin{align*}
(d\alpha)^\natural=\Phi_\#(d\alpha) = d(\Phi_\#\alpha)=d\left(\alpha^\natural\right).
\end{align*}

To prove the identity for $\delta$, let us introduce the sign $\sigma=(-1)^{n(k+1)+1}$ (which is actually immaterial to the proof). It then holds that $$\delta=\sigma\star d \star,$$ where $\star$ is the Hodge star  on $(\half,\gc)$. As before, denote $y = \Phi(x)$ and $\alpha = \sum \alpha_{i_1\ldots i_\ell}dx^{i_1}\wedge\cdots\wedge dx^{i_k}$ in a fixed local co-ordinate system.

Let us first compute $\delta\left(\alpha^\natural\right)$. In local co-ordinates we have
\begin{align*}
\alpha^\natural = \sum\left(\alpha_{i_1\ldots i_\ell}\circ\Phi^{-1}\right) \left(\frac{\p\Phi^{j_1}}{\p x^{i_1}}\circ\Phi^{-1}\right)\cdots\left(\frac{\p\Phi^{j_k}}{\p x^{i_k}}\circ\Phi^{-1}\right)
dy^{j_1} \wedge \cdots \wedge dy^{j_k}.
\end{align*}
Let $dy^{\gamma_1}\wedge \cdots \wedge dy^{\gamma_{n-k}}$ be a simple $(n-k)$-form such that 
\begin{align*}
\left(dy^{\gamma_1}\wedge \cdots \wedge dy^{\gamma_{n-k}}\right)\wedge\left(dy^{j_1} \wedge \cdots \wedge dy^{j_k}\right)
\end{align*}
is the canonical unit volume form on $(\half,\gc)$. Then 
\begin{align*}
\star\alpha^\natural = \sum\left(\alpha_{i_1\ldots i_\ell}\circ\Phi^{-1}\right) \left(\frac{\p\Phi^{j_1}}{\p x^{i_1}}\circ\Phi^{-1}\right)\cdots\left(\frac{\p\Phi^{j_k}}{\p x^{i_k}}\circ\Phi^{-1}\right)
dy^{\gamma_1}\wedge \cdots \wedge dy^{\gamma_{n-k}}.
\end{align*}
Taking another exterior differential gives us
\begin{align*}
d\star\alpha^\natural = \frac{\p}{\p y^m}\left\{\left(\alpha_{i_1\ldots i_\ell}\circ\Phi^{-1}\right) \left(\frac{\p\Phi^{j_1}}{\p x^{i_1}}\circ\Phi^{-1}\right)\cdots\left(\frac{\p\Phi^{j_k}}{\p x^{i_k}}\circ\Phi^{-1}\right)\right\}dy^m \wedge 
dy^{\gamma_1}\wedge \cdots \wedge dy^{\gamma_{n-k}}.
\end{align*}
Finally, taking $\sigma$ and another Hodge star, we arrive at
\begin{align*}
\delta\left(\alpha^\natural\right) = \sigma (-1)^{\ell}\frac{\p}{\p y^{j_\ell}}\left\{\left(\alpha_{i_1\ldots i_\ell}\circ\Phi^{-1}\right) \left(\frac{\p\Phi^{j_1}}{\p x^{i_1}}\circ\Phi^{-1}\right)\cdots\left(\frac{\p\Phi^{j_k}}{\p x^{i_k}}\circ\Phi^{-1}\right)\right\} dy^{j_1} \wedge \cdots \wedge \widehat{dy^{j_\ell}}\wedge\cdots\wedge dy^{j_k}.
\end{align*}
Here $\widehat{dy^{j_\ell}}$ means that the component $dy^{j_\ell}$ has been omitted; \emph{ditto}. Note that $y^{j_\ell}$ is nothing but $\Phi^{j_\ell}$; one may thus conclude that 
\begin{align}\label{delta, long 1}
\delta\left(\alpha^\natural\right) &= \sigma (-1)^{\ell}\sum\left(\frac{\p\alpha_{i_1\ldots i_\ell}}{\p y^{j_\ell}}\circ\Phi^{-1}\right) \left(\frac{\p\Phi^{j_1}}{\p x^{i_1}}\circ\Phi^{-1}\right)\cdots\widehat{\left(\frac{\p\Phi^{j_\ell}}{\p x^{i_\ell}}\circ\Phi^{-1}\right)}\cdots\left(\frac{\p\Phi^{j_k}}{\p x^{i_k}}\circ\Phi^{-1}\right)\nonumber\\ &\qquad dy^{j_1} \wedge \cdots \wedge \widehat{dy^{j_\ell}}\wedge\cdots\wedge dy^{j_k}.
\end{align}

On the other hand,  $\left(\delta\alpha\right)^\natural$ can be computed as follows. Let $dx^{\beta_1}\wedge \cdots \wedge dx^{\beta_{{n-k}}}$ be a simple $(n-k)$-form such that 
\begin{align*}
(dx^{\beta_1}\wedge \cdots \wedge dx^{\beta_{{n-k}}}) \wedge (dx^{i_1}\wedge\cdots\wedge dx^{i_k})
\end{align*}
is the canonical unit volume form on the domain manifold. Then it holds that
\begin{align*}
\star \alpha = \sum{\alpha_{i_1\cdots i_k}}dx^{\beta_1}\wedge \cdots \wedge dx^{\beta_{{n-k}}}
\end{align*}
and that
\begin{align*}
d\star\alpha = \sum \frac{\p \alpha_{i_1\cdots i_k}}{\p x^q} dx^q \wedge dx^{\beta_1}\wedge \cdots \wedge dx^{\beta_{{n-k}}}.
\end{align*}
Taking $\sigma$ and another Hodge star gives us
\begin{align*}
\delta\alpha = \sigma (-1)^{\ell} \sum\frac{\p \alpha_{i_1\cdots i_k}}{\p x^{i_\ell}}dx^{i_1}\wedge\cdots\wedge \widehat{dx^{i_\ell}}\wedge \cdots \wedge dx^{i_k}.
\end{align*}

Therefore, the pushforward $\Phi_\#$ sends $\delta\alpha$ to $\delta\left(\alpha^\natural\right)$ in view of Eq.~\eqref{delta, long 1} and the chain rule, once we relabel the indices $i_1, \ldots, i_k$ as  $j_1, \ldots, j_k$.    \end{proof}

\begin{corollary}\label{cor: d, delta commute with natural}
For any differential $k$-forms on the chart $U$, it holds that
\begin{equation*}
\left( d \left(\alpha^\natural\right),d \left(\beta^\natural\right)\right)_{\gc} = \left(\left(d \alpha\right)^\natural,\left(d \beta\right)^\natural\right)_{\gc},\qquad
\left(\delta \left(\alpha^\natural\right),\delta \left(\beta^\natural\right)\right)_{\gc} = \left(\left(\delta \alpha\right)^\natural,\left(\delta \beta\right)^\natural\right)_{\gc}.
\end{equation*}
In fact, the above identities hold whenever the pairings are well defined in the distributional sense. \end{corollary}
\begin{proof}
This is an immediate consequence of Lemma~\ref{lem: compare diff op}.
\end{proof}

Now we are at the stage of proving Proposition~\ref{prop: chart}:

\begin{proof}[Proof of Proposition~\ref{prop: chart}]
Given $\kappa>0$ and the manifold-with-boundary $(\M,g_0)$, we choose the chart $U$ as before. (See Remark~\ref{rem: C1 geometry} together with the estimates in Eqs.~\eqref{metric est}, \eqref{inverse metric est}, \eqref{det est}, and  \eqref{christoffel est}; the key is that the metric $g_0$ is $\mathcal{O}(\kappa)$-close on $U$ to the constant metric $g_0(P)$ in the $C^1$-topology.)

Let $u \in W^{1,q}_{\diamondsuit}\Omega^k(\M)$ be any $k$-form supported in $U \subset \clM$ with $q \in ]1,\infty[$.

Under the assumption~\eqref{assumption, chart}, namely that 
\begin{align*}
\sup\left\{\frac{\left| ( d u, d \phi)_{g_0} +( \delta u, \delta \phi)_{g_0} + ( u,\phi)_{g_0} \right|}{\|d\phi\|_{L^{r'}(U,g_0)} + \|\delta\phi\|_{L^{r'}(U,g_0)}+\|\phi\|_{L^{r'}(U,g_0)}}  :\,\phi \in C^\infty_\diamondsuit\Omega^k(U) \right\} < \infty
\end{align*}
for some $r \in ]1,\infty[$, we shall prove for both cases $\diamondsuit=N$ and $D$ at the same stroke.

First, let us estimate the numerator
\begin{align*}
\ll u,\phi\gg_{g_0}:= ( d u, d \phi)_{g_0} +( \delta u, \delta \phi)_{g_0} + ( u,\phi)_{g_0}
\end{align*}
from below. To this end, we consider the difference
\begin{align*}
D\{u,\phi\}:=\underbrace{\ll u,\phi\gg_{g_0}-\ll u,\phi\gg_{g_0(P)}}_{D_1} + \underbrace{\ll u,\phi\gg_{g_0(P)}-\ll u^\natural,\phi^\natural\gg_{\gc}}_{D_2},
\end{align*}
where as before $\gc := \Phi_\#\left( g_0(P) \right)$, and $\Phi$ is the co-ordinate map on the chart $U$. By Lemma~\ref{lem: bdry cond} we have $|D_2| \leq C_{9}\kappa$, and Eq.~\eqref{C9} implies $|D_1| \leq C_{9}\kappa$, where $C_{9}$ depends only on $k$, $n$, and the $C^1$-geometry of $U$. Thus
\begin{equation}\label{x1}
\left|\ll u,\phi\gg_{g_0}\right| \geq \left|\ll u^\natural,\phi^\natural\gg_{\gc}\right| - C_{10}\kappa,
\end{equation}
where $C_{10}$ has the same dependence.

Next we estimate the denominator 
\begin{align*}
|||\phi|||_{r'; g_0} := \|d\phi\|_{L^{r'}(U,g_0)} + \|\delta\phi\|_{L^{r'}(U,g_0)}+\|\phi\|_{L^{r'}(U,g_0)}
\end{align*}
from above. For this purpose we recall that
\begin{align*}
\|\alpha\|_{L^{r'}(U,g_0)} := \left\{\int_U \left(\langle\alpha,\alpha\rangle_{g_0}\right)^{\frac{r'}{2}} \,\dd V_{g_0}\right\}^{\frac{1}{r'}}
\end{align*}
for any differential $\ell$-form $\alpha$ on $U$. A straightforward adaptation of the proof for Lemma~\ref{lem: inner prod comparison} (see, in particular, Eq.~\eqref{C9} therein) gives us
\begin{align*}
\left|\|\alpha\|_{L^{r'}(U,g_0)} - \left\|\alpha^\natural\right\|_{L^{r'}(\Phi(U),\gc)}\right|\leq C_{11}\kappa,
\end{align*}
where $C_{11}$ depends only on $\ell$, $r'$, $n$, and the $C^1$-geometry of $U$. Taking $\alpha=d\phi$, $\delta\phi$, and $\phi$, respectively, we obtain
\begin{align}\label{x2}
|||\phi|||_{r',g_0} \leq \left|\left|\left|\phi^\natural\right|\right|\right|_{r',\gc} + C_{12}\kappa,
\end{align}
where $C_{12}$ depends only on $k$, $r$, $n$, and the $C^1$-geometry of $U$. 

Putting together Eqs.~\eqref{x1} and \eqref{x2},  recalling that $\phi^\natural$ in $C^\infty_\diamondsuit\Omega^k(\half)$ by Lemma~\ref{lem: bdry cond}, and invoking Corollary~\ref{cor: d, delta commute with natural} (with $\varphi \equiv \phi^\natural$), we arrive at the bounds
\begin{align*}
&\infty>\sup\left\{\frac{\left|( d u, d \phi)_{g_0} +(\delta u, \delta \phi)_{g_0} + ( u,\phi)_{g_0} \right|}{\|d\phi\|_{L^{r'}(U,g_0)} + \|\delta\phi\|_{L^{r'}(U,g_0)}+\|\phi\|_{L^{r'}(U,g_0)}}  :\,\phi \in C^\infty_\diamondsuit\Omega^k(U) \right\}\\
&\quad \geq \sup\left\{\frac{\left| \left( d u^\natural, d \varphi\right)_{\gc} +\left( \delta u^\natural, \delta \varphi\right)_{\gc} + \left( u^\natural,\varphi\right)_{g_0} \right|-C_{10}\kappa}{\|d\varphi\|_{L^{r'}(\Phi(U),\gc)} + \|\delta\varphi\|_{L^{r'}(\Phi(U),\gc)}+\|\varphi\|_{L^{r'}(\Phi(U),\gc)}+C_{12}\kappa}  :\,\varphi \in C^\infty_\diamondsuit\Omega^k(\half) \right\}\\
&\quad \geq C_{13}  \left(\sup\left\{\frac{\left| \left( d u^\natural, d \varphi\right)_{\gc} +\left( \delta u^\natural, \delta \varphi\right)_{\gc} + \left( u^\natural,\varphi\right)_{g_0} \right|}{\|d\varphi\|_{L^{r'}(\Phi(U),\gc)} + \|\delta\varphi\|_{L^{r'}(\Phi(U),\gc)}+\|\varphi\|_{L^{r'}(\Phi(U),\gc)}}  :\,\varphi \in C^\infty_\diamondsuit\Omega^k(\half) \right\}-\kappa\right),
\end{align*}
where $C_{13}$ depends only on $C_{10}$ and $C_{12}$.

Then, one may refer to Corollary~\ref{cor: halfspace, constant metric} to deduce that $u^\natural \in W^{1,r}\Omega^k(\half,\gc)$. Moreover, it satifies the bound
\begin{align*}
\left\|u^\natural\right\|_{W^{1,r}(\half,\gc)} \leq C_{14}\sup\left\{\frac{\left| \left( d u^\natural, d \varphi\right)_{\gc} +\left( \delta u^\natural, \delta \varphi\right)_{\gc} + \left( u^\natural,\varphi\right)_{g_0} \right|}{\|d\varphi\|_{L^{r'}(\Phi(U),\gc)} + \|\delta\varphi\|_{L^{r'}(\Phi(U),\gc)}+\|\varphi\|_{L^{r'}(\Phi(U),\gc)}}  :\,\varphi \in C^\infty_\diamondsuit\Omega^k(\half) \right\},
\end{align*}
with $C_{14}$ depending on nothing but $k$, $r$, $n$, and the $C^1$-geometry of $U$. 

On the other hand, the same proof for Eq.~\eqref{x2} (recalling the definition of $|||\bullet|||_{r',g_0}$ and $|||\bullet|||_{r',\gc}$) also gives us
\begin{align*}
\|u\|_{W^{1,r}(U,g_0)} \leq \left\|u^\natural\right\|_{W^{1,r}(\Phi(U),\gc)} + C_{15}\kappa,
\end{align*}
where $C_{15}$ has the same dependence as $C_{14}$. We thus arrive at
\begin{align}\label{C16}
\|u\|_{W^{1,r}(U,g_0)} &\leq  C_{14}\sup\left\{\frac{\left| \left( d u^\natural, d \varphi\right)_{\gc} +\left( \delta u^\natural, \delta \varphi\right)_{\gc} + \left( u^\natural,\varphi\right)_{g_0} \right|}{\|d\varphi\|_{L^{r'}(\Phi(U),\gc)} + \|\delta\varphi\|_{L^{r'}(\Phi(U),\gc)}+\|\varphi\|_{L^{r'}(\Phi(U),\gc)}}  :\,\varphi \in C^\infty_\diamondsuit\Omega^k(\half) \right\} +C_{15}\kappa\nonumber\\
&\leq \frac{C_{14}}{C_{13}}\left(S[u] +\kappa \right) +C_{15}\kappa\nonumber\\
&\leq C_{16} \left(S[u] +\kappa \right).
\end{align}
Here $C_{16}$ depends again only on $k$, $n$, $r$, and the $C^1$-geometry of $U$, and we adopt the notation
\begin{align*}
S[u]:=\sup\left\{\frac{\left| ( d u, d \phi)_{g_0} +( \delta u, \delta \phi)_{g_0} +( u,\phi)_{g_0} \right|}{\|d\phi\|_{L^{r'}(U,g_0)} + \|\delta\phi\|_{L^{r'}(U,g_0)}+\|\phi\|_{L^{r'}(U,g_0)}}  :\,\phi \in C^\infty_\diamondsuit\Omega^k(U) \right\}.
\end{align*}

We \emph{claim} that one may assume $\kappa \leq S[u]$. Indeed, since $\kappa$ and $C_{16}$ are both independent of $u$, we can scale $u \mapsto \lambda u$ in  Eq.~\eqref{C16} to get $\|\lambda u\|_{W^{1,r}(U,g_0)} \leq C_{16}(\lambda S[u] + \kappa)$ for any $\lambda >0$. Hence, were the \emph{claim} false, one must have had $S[u]=0$. But then Eq.~\eqref{C16} would imply that $\|u\|_{W^{1,r}(U,g_0)} \leq C_{16} \kappa$ for arbitrarily small $\kappa$, so $u \equiv 0$.

Therefore, the above \emph{claim} together with Eq.~\eqref{C16} yields that
\begin{align*}
\|u\|_{W^{1,r}(U,g_0)} \leq 2\cdot C_{16}\sup\left\{\frac{\left|( d u, d \phi)_{g_0} +(\delta u, \delta \phi)_{g_0} + ( u,\phi)_{g_0} \right|}{\|d\phi\|_{L^{r'}(U,g_0)} + \|\delta\phi\|_{L^{r'}(U,g_0)}+\|\phi\|_{L^{r'}(U,g_0)}}  :\,\phi \in C^\infty_\diamondsuit\Omega^k(U) \right\}<\infty.
\end{align*}
In particular, by assumption~ \eqref{assumption, chart} one has $u \in W^{1,r}\Omega^k(U,g_0)$. The proof is now complete.  \end{proof}

\subsection{Proof of Theorem~B}\label{subsec: proof of thm B}

Finally, we are in the situation of concluding the proof of Theorem~B. The cases $\diamondsuit=D$ and $N$ shall the treated simultaneously.

Since $\M$ is bounded, when $q \geq r$ there is nothing to prove (as $L^q\Omega^k(\M) \subset L^r\Omega^k(\M)$). So we assume throughout in the sequel that 
\begin{equation}\label{range of indices}
1<q<r<\infty. 
\end{equation}

Let $\kappa_0$ be an arbitrarily small positive number. Since the manifold-with-boundary $(\M,g_0)$ has $C^1$-bounded geometry, one may select $\mathsf{Atlas}$, a \emph{``$\kappa_0$-good'' atlas}  as follows. $\mathsf{Atlas}$ consists of charts $\{U_i\}_{i=0}^N$ such that
\begin{align}\label{good chart, 1}
\bigcup_{i=1}^N U_i\supset \p\M
\end{align}
and that for an point $P \in U_i$, one has
\begin{equation}\label{good chart, 2}
\sup_{1\leq i \leq N}\left\{
\left\|g_0 - g_0(P)\right\|_{C^0(U_i)} + \left\|Dg_0 - [Dg_0](P)\right\|_{C^0(U_i)}  \right\}\leq \kappa_0.
\end{equation}
The former condition~\eqref{good chart, 1} means that $U_0$ is an interior chart of $\M$ and $U_1,\ldots,U_N$ are boundary charts; the latter condition~\eqref{good chart, 2} is a restatement of Remark~\ref{rem: C1 geometry}.

Next, let $\{\chi_i\}_{i=0}^N$ be a differentiable partition of unity subordinate to $\mathsf{Atlas}$. The quantity
\begin{align*}
\sup_{0\leq i \leq N} \|\chi_i\|_{C^1(\M,g_0)} 
\end{align*}
is finite, as it can be controlled by the $C^1$-geometry of $(\M,g)$. Consider
\begin{equation*}
u_j := \chi_j u\qquad \text{ for } j \in \{1,2,\ldots,N\}.
\end{equation*}
We check that for each such $j$ there holds $u_j \in W^{1,r}\Omega^k(\M)$. As a remark, by construction $u_j$ is compactly supported in the boundary chart $U_j$. 

We proceed with straightforward computations. Note first that
\begin{align*}
(d u_j, d \phi)_{g_0} = (d\chi_j \wedge u,d\phi)_{g_0} + \chi_j ( d u, d \phi)_{g_0},
\end{align*}
where $d\chi_j \wedge u$ is the wedge product of a $1$-form with an $r$-form; similarly we have
\begin{equation*}
(\delta u_j, \delta \phi)_{g_0} = \sigma\left(\star\left[ d\chi_j \wedge \star u \right],d\phi \right)_{g_0} +\chi_j (\delta u, \delta \phi)_{g_0},
\end{equation*}
where $\star$ is the Hodge star operator and $\sigma=(-1)^{n(k+1)+1}$. Hence,
\begin{align*}
&\frac{(d u_j, d \phi)_{g_0} +( \delta u_j, \delta \phi)_{g_0} + ( u_j,\phi)_{g_0} }{\|d\phi\|_{L^{r'}(U_j,g_0)} + \|\delta\phi\|_{L^{r'}(U_j,g_0)}+\|\phi\|_{L^{r'}(U_j,g_0)}} \\
&\quad = \chi_j\left(\frac{ (d u, d \phi)_{g_0} +( \delta u, \delta \phi)_{g_0} + (u,\phi)_{g_0} }{\|d\phi\|_{L^{r'}(U_j,g_0)} + \|\delta\phi\|_{L^{r'}(U_j,g_0)}+\|\phi\|_{L^{r'}(U_j,g_0)}} \right) +\frac{ (d\chi_j \wedge u,d\phi)_{g_0} + \sigma\left(\star\left[ d\chi_j \wedge \star u \right],d\phi \right)_{g_0}}{\|d\phi\|_{L^{r'}(U_j,g_0)} + \|\delta\phi\|_{L^{r'}(U_j,g_0)}+\|\phi\|_{L^{r'}(U_j,g_0)}} \\
&\quad= :I_1+ I_{2},
\end{align*}
for arbitrary test $k$-form $\phi$ compactly supported in $U_j$.

Now we turn to the estimation for $I_1$, $I_2$. We consider two cases separately. Recall
$\dim\M=n$, that $u$ is a $k$-form with finite $W^{1,q}$-norm, and that
\begin{equation*}
\mathfrak{S}[u] := \sup\left\{\frac{\left|( d u, d \phi)_{g_0} +(\delta u, \delta \phi)_{g_0} + ( u,\phi)_{g_0} \right|}{\|d\phi\|_{L^{r'}(\M,g_0)} + \|\delta\phi\|_{L^{r'}(\M,g_0)}+\|\phi\|_{L^{r'}(\M,g_0)}}  :\,\phi \in C^\infty_\diamondsuit\Omega^k(\M) \right\}<\infty.
\end{equation*}

\smallskip
\noindent
\underline{Case 1: $q \geq n$}. We have the standard Sobolev--Morrey embedding $W^{1,q}(U_j) \emb L^r(U_j)$; so
\begin{align*}
|I_2| &\leq C_{17}\frac{\|u\|_{L^r(U_j, g_0)}\|d\phi\|_{L^{r'}(U_j, g_0)}}{{\|d\phi\|_{L^{r'}(U_j,g_0)} + \|\delta\phi\|_{L^{r'}(U_j,g_0)}+\|\phi\|_{L^{r'}(U_j,g_0)}}} \\
&\leq C_{18} \|u\|_{W^{1,q}(U_j, g_0)},
\end{align*}
where $C_{17}$ and $C_{18}$ depend only on $r$, $k$, $n$, and the $C^1$-geometry of $U_j$.

On the other hand, it is clear that $|I_1| \leq \mathfrak{S}[u]$, for which purpose we consider the extension-by-zero of $\phi \in C^\infty_{c}\Omega^k(U_j)$ to a test form on $\M$. In this way we have
\begin{align*}
\left|\frac{( d u_j, d \phi)_{g_0} +(\delta u_j, \delta \phi)_{g_0} + ( u_j,\phi)_{g_0} }{\|d\phi\|_{L^{r'}(U_j,g_0)} + \|\delta\phi\|_{L^{r'}(U_j,g_0)}+\|\phi\|_{L^{r'}(U_j,g_0)}} \right| \leq C_{18} \|u\|_{W^{1,q}(U_j, g_0)} + \mathfrak{S}[u] < \infty.
\end{align*}

But $u_j$ is supported on a sufficiently small chart, hence we may apply  
Proposition~\ref{prop: chart} to infer that
\begin{align*}
u_j \equiv \chi_j u \in W^{1,r}\Omega^k(U_j, g_0),
\end{align*}
 together with the estimate
\begin{align*}
\|u_j\|_{W^{1,r}(U_j,g_0)} \leq C_{19} \left( C_{18} \|u\|_{W^{1,q}(U_j, g_0)} + \mathfrak{S}[u]\right),
\end{align*}
where $C_{19}$ depends again only on $r$, $k$, $n$, and the $C^1$-geometry of $U_j$. Thus
\begin{align*}
\|u\|_{W^{1,r}(\M,g_0)} \leq (N+1)C_{19} \left( C_{18} \|u\|_{W^{1,q}(U_j, g_0)} + \mathfrak{S}[u]\right).
\end{align*}
Here $(N+1)$ is the number of charts in $\mathsf{Atlas}$, which can be controlled in turn by the $C^1$-geometry of $(\M,g_0)$.

To conclude the proof for Case 1, it remains to note that the quantitative statement~\eqref{r-est on chart, lemma} in Proposition~\ref{prop: chart} allows us to bound (with $C_{20}$ having the same dependence with $C_{19}$) that
\begin{align*}
&\|u\|_{W^{1,r}(U_j,g_0)}\\
 \leq\,& C_{20}\sup\left\{\frac{\left|( d u, d \phi)_{g_0} +(\delta u, \delta \phi)_{g_0} + ( u,\phi)_{g_0} \right|}{\|d\phi\|_{L^{r'}(U_j,g_0)} + \|\delta\phi\|_{L^{r'}(U_j,g_0)}+\|\phi\|_{L^{r'}(U_j,g_0)}}  :\,\phi \in C^\infty_\diamondsuit\Omega^k(\M),\,{\rm supp}\,\phi \Subset U_j \right\}.
\end{align*}
Therefore, we get
\begin{equation}\label{final bound, A}
\|u\|_{W^{1,r}(\M,g_0)} \leq (N+1)C_{19} \left( C_{18} C_{20}+1\right)\mathfrak{S}[u] <\infty.
\end{equation}

\smallskip
\noindent
\underline{Case 2: $1<q < n$}. In this case we shall follow the clever argument on p.1882 of Kozono--Yanagisawa \cite{ky}, which allows us to raise the regularity of $u$ from $W^{1,q}$ to $W^{1,s(q)}$ and then to $W^{1,s(s(q))}$, so on and so forth; here and hereafter, for $p \in ]1,\infty[$, $s(p)$ denotes the Sobolev conjugate of $p$. In each step it holds that
\begin{align*}
\overbrace{s\circ\ldots \circ s}^{(m+1)\text{ times}} (q) - 
\overbrace{s\circ\ldots \circ s}^{m\text{ times}} (q) \geq \e_0 >0,
\end{align*}
where $\e_0$ is independent of $m$. Thus in finitely many steps we shall return to Case 1. This approach is reminiscent of the ``LIR'' (local increasing regularity) method due to Amar (\cite{a1, a2}).

Note first that, again, we trivially have $|I_1| \leq \mathfrak{S}[u]$. To control $I_2$, consider as before
\begin{align*}
|I_2| \leq C_{21} \|u\|_{L^r(U_j,g_0)} \leq C_{22}\|u\|_{W^{1,\frac{nr}{n+r}}(U_j,g_0)}.
\end{align*}
Here $\frac{nr}{n+r}$ is the Sobolev predual for $r$. Therefore, all the arguments in Case~1 carry through whenever $s(r) \geq q$, with the constants again depending only on $r$, $n$, $k$, and the $C^1$-geometry of $\M$. 

This proves that
\begin{align}\label{scheme}
u \in \widehat{W^{1,q}_\diamondsuit}\Omega^k(\M) \quad \Longrightarrow \quad u \in \widehat{W^{1,s(q)}_\diamondsuit}\Omega^k(\M),
\end{align}
where 
$s(q) := \frac{nq}{n-q}$. Thus $\overbrace{s\circ\ldots\circ s}^{m\text{ times}}(q) = \frac{nq}{n-mq}$, which gives us
\begin{align*}
\overbrace{s\circ\ldots\circ s}^{(m+1)\text{ times}}(q) -\overbrace{s\circ\ldots\circ s}^{m\text{ times}}(q) =\frac{nq^2}{(n-mq)(n-mq-q)} \geq \e_0(q,n)>0,
\end{align*}
with $\e_0$ being independent of $m$. For example, clearly one may take $\e_0=n^{-1}q^2$.

As a consequence, we can now  continue the implication in Eq.~\eqref{scheme} by
\begin{align*}
u \in \widehat{W^{1,q}_\diamondsuit}\Omega^k(\M) &\, \Rightarrow \, u \in \widehat{W^{1,s(q)}_\diamondsuit}\Omega^k(\M) \\
&\, \Rightarrow \, u \in \widehat{W^{1,s(s(q))}_\diamondsuit}\Omega^k(\M)  \\
&\,  \Rightarrow \cdots\Rightarrow \, u \in
\widehat{W^{1,{\underbrace{s\circ\ldots\circ s}_{m\text{ times}}(q)}}_\diamondsuit}\Omega^k(\M)\subset\widehat{W^{1,r}_\diamondsuit}\Omega^k(\M)
\end{align*}
once we choose $m$ so large that $mq^2/n \geq r$, thanks to the compactness of $\M$. Then we may proceed as in Case 1. The proof of Theorem~B is now complete.

\subsection{A remark}
For subsequent developments, we note that Theorem~B remains to hold if $\|d\phi\|_{L^{r'}}+ \|\delta\phi\|_{L^{r'}}$ in the denominator is replaced by $\|\na \phi\|_{L^{r'}}$, where $\na$ is the covariant derivative obtained from the Levi-Civita connection.

\begin{corollary}\label{cor: main}
Let $(\M,g_0)$ be an $n$-dimensional compact Riemannian manifold-with-boundary, let $\diamondsuit$ denote either $D$ or $N$,  let $r$ and $q$ be any numbers in $]1,\infty[$, and let $k\in\{0,1,2,\ldots,n\}$. Assume $u \in W^{1,q}_\diamondsuit\Omega^k(\M,g_0)$ and that the quantity below is finite:
\begin{equation*}
\sup\left\{\frac{\left| ( d u, d \phi)_{g_0} +( \delta u, \delta \phi)_{g_0} + ( u,\phi)_{g_0} \right|}{\|\na\phi\|_{L^{r'}(\M,g_0)}+\|\phi\|_{L^{r'}(\M,g_0)}}:\,\phi \in C^\infty_\diamondsuit\Omega^k(\M) \right\}.
\end{equation*}
Then $u \in  W^{1,r}_\diamondsuit\Omega^k(\M)$.

 In addition, the following variational inequality is satisfied:
\begin{equation*}
\|u\|_{W^{1,r}} \leq C\sup \left\{\frac{\left|( d u, d \phi)_{g_0} +(\delta u, \delta \phi)_{g_0} + ( u,\phi)_{g_0} \right|}{\|\na\phi\|_{L^{r'}(\M,g_0)}  +\|\phi\|_{L^{r'}(\M,g_0)}}:\,\phi \in C^\infty_\diamondsuit\Omega^k(\M)  \right\},
\end{equation*}
where $C$ depends only on $r$, $n$, $k$, and the $C^1$-geometry of $(\M,g_0)$.
\end{corollary}


\begin{proof}
	All the arguments in \S\S\ref{subsec a1}, \ref{subsec a2}, \ref{subsec a3}, and \ref{subsec: proof of thm B} go through, as long as we make the following modifications: In the proof of Lemma~\ref{lem: whole space}, we change the inequality~\eqref{new 1} to
	\begin{align*}
&\sup\left\{\frac{\left| (d u, d \phi) +(\delta u, \delta \phi) +(u,\phi) \right|}{\|\na\phi\|_{L^{r'}(\R^n)}+\|\phi\|_{L^{r'}(\R^n)}}  :\,\phi \in C^\infty_{c}\Omega^k(\R^n) \right\} \\
&\quad \geq 
\sup\left\{\frac{\left| ( d u, d \na_j\psi) +( \delta u, \delta \na_j\psi)+ ( u,\na_j\psi) \right|}{\|\na\na_i\psi\|_{L^{r'}(\R^n)}+\|\na_j\psi\|_{L^{r'}(\R^n)}}  :\,\psi \in C^\infty_{c}\Omega^k(\R^n) \right\},
\end{align*}
and still use the Calder\'{o}n--Zygmund estimate to bound $\|\na\na_i\psi\|_{L^{r'}(\R^n)}\lesssim \|\Delta\psi\|_{L^{r'}(\R^n)} + \|\psi\|_{L^{r'}(\R^n)}$. For the proof of Lemma~\ref{lem: halfspace}, we only need to observe that Eq.~\eqref{new2} remains to hold for $T=\na$. In addition, the analogue of Lemma~\ref{lem: compare diff op}, namely that $(\na \alpha)^\natural = \na\left( \alpha^\natural\right)$, holds from the definition of pullback connections; also see Eq.~\eqref{christoffel est} for estimates on the difference of connections. Thus the arguments for Corollary~\ref{cor: halfspace, constant metric}, Proposition~\ref{prop: chart}, and Theorem~B carry over almost verbatim, once we change each occurrence of $\|d\phi\|_{L^{r'}} +\|\delta\phi\|_{L^{r'}}$ to $\|\na\phi\|_{L^{r'}}$.   \end{proof}

Also, taking $q=r$ in Theorem~B and Corollary~\ref{cor: main}, we immediately get the following estimation for the $W^{1,r}$-norm of $u$ by duality:
\begin{proposition}\label{prop: W1,r estimate}
Let $(\M,g_0)$ be an $n$-dimensional compact Riemannian manifold-with-boundary, let $k \in \{0,1,2,\ldots,n\}$, and let $r \in ]1,\infty[$ be arbitrary. Assume that $u \in \widehat{W^{1,r}_\diamondsuit}\Omega^k(\M)$ for either $\diamondsuit=D$ or $N$. Then
\begin{equation*}
\|u\|_{W^{1,r}(\M,g_0)} \leq C\sup \left\{\frac{\left| ( d u, d \phi)_{g_0} +( \delta u, \delta \phi)_{g_0} + ( u,\phi)_{g_0} \right|}{\|d\phi\|_{L^{r'}(\M,g_0)} + \|\delta\phi\|_{L^{r'}(\M,g_0)}+\|\phi\|_{L^{r'}(\M,g_0)}} :\,\phi \in C^\infty_{\diamondsuit}\Omega^{k}(\M) \right\},
\end{equation*}
where $C$ depends only on $r$, $n$, $k$, and the $C^1$-geometry of $(\M,g_0)$.

There is another constant $C'$ with the same dependency as $C$ such that 
\begin{equation*}
\|u\|_{W^{1,r}(\M,g_0)} \leq C'\sup \left\{\frac{\left| (d u, d \phi)_{g_0} +( \delta u, \delta \phi)_{g_0} + ( u,\phi)_{g_0} \right|}{\|\na\phi\|_{L^{r'}(\M,g_0)} +\|\phi\|_{L^{r'}(\M,g_0)}} :\,\phi \in C^\infty_{\diamondsuit}\Omega^{k}(\M) \right\}.
\end{equation*}
\end{proposition}

\section{Proof of the Gaffney's inequality}\label{sec: gaffney}

The goal of this section is to prove Theorem~A (reproduced below).

\begin{theorem*}
Let $(\M,g_0)$ be an $n$-dimensional compact Riemannian manifold-with-boundary, let $k \in \{0,1,2,\ldots,n\}$, and let $r \in ]1,\infty[$ be arbitrary. Consider $u \in \widehat{W^{1,r}_D}\Omega^k(\M)$  or $\widehat{W^{1,r}_N}\Omega^k(\M)$, a differential $k$-form subject to designated boundary conditions. Then
\begin{equation*}
\|u\|_{W^{1,r}(\M)} \leq C\left\{\|du\|_{L^r(\M)} + \|\delta u\|_{L^r(\M)} + \|u\|_{L^r(\M)}\right\},
\end{equation*}
where $C$ depends only on $r$, $k$, $n$, and the geometry of $\M$.
\end{theorem*}

An immediate consequence is as follows:

\begin{corollary}\label{cor: Cr = C-hat r}
Let $(\M,g_0)$ be an $n$-dimensional compact Riemannian manifold-with-boundary, let $k \in \{0,1,2,\ldots,n\}$, and let $r \in ]1,\infty[$ be arbitrary. Then $$\widehat{W^{1,r}_\diamondsuit}\Omega^k(\M)=W^{1,r}\Omega^k(\M)$$ for either $\diamondsuit=D$ or $N$.
\end{corollary}

\begin{proof}[Proof of Corollary~\ref{cor: Cr = C-hat r}]
$\widehat{W^{1,r}_\diamondsuit}\Omega^k(\M)\subset W^{1,r}\Omega^k(\M)$ is trivial. On the other hand, knowing that $C^\infty_{\diamondsuit}\Omega^k(\M)$ is dense in $\CC^r_k(\M)$, the other inclusion would follow immediately from the Gaffney's inequality (Theorem~A) and a density argument.

To show the density result, by considering standard bases of simple $k$-forms as in the proof of Lemma~\ref{lem: whole space}, it suffices to take $k=0$. By partition of unity we can reduce to the case where $\M$ is covered by a single co-ordinate chart. We may also take this chart to be Euclidean via the co-ordinate map. Then it follows from classical results as in  Duvaut--Lions \cite{dl}, Chapter~7, Lemmata
4.2 and 6.1. (See also Corollaries~3.2 and 3.3 in Iwaniec--Scott--Stroffolini  \cite{iss}.)  \end{proof}


\begin{proof}[Proof of Theorem~A]

First, observe that we only need to prove for the case $\diamondsuit=D$, namely that $\ttt u = 0 = \ttt \phi$. Indeed, if $u,\phi \in C^\infty_N\Omega^k(\M)$, then their Hodge duals $\star u$ and $\star \phi \in C^\infty_D\Omega^k(\M)$; so Eq.~\eqref{gaffney} for $\star u$ and $\star \phi$ implies the same result (with the same constants) for $u$ and $\phi$, as the Hodge star operator is an $L^r$-isometry for any $r \in \one$. 

Now we assume $\ttt u = 0 = \ttt \phi$. Integration by parts and the Stokes' theorem give us
\begin{align*}
\mathscr{D}(u,\phi) &:= \int_\M \langle du, d\phi\rangle_{\bigwedge^{k+1}} \,\dvg + \int_\M \langle \delta u, \delta \phi\rangle_{\bigwedge^{k-1}} \,\dvg \\
&=\int_\M \langle \Delta\phi, u\rangle_{\bigwedge^k}\,\dvg + \int_{\p\M} \ttt u \wedge \star \nnn(d\phi) - \int_{\p\M} \ttt(\delta \phi) \wedge \star \nnn u.
\end{align*}
See p.62, Corollary~2.1.4 in \cite{s} for details. Recall that the pairing $\langle\bullet,\bullet\rangle_{\bigwedge^p}$ denotes the pointwise inner product for differential $p$-forms on $\M$ induced by the metric $g_0$. As $\ttt u=0$, we have
\begin{align*}
\int_{\p\M} \ttt u \wedge \star \nnn(d\phi)=0.
\end{align*}

Consider now $$I:=\int_\M \langle \Delta\phi, u\rangle_{\bigwedge^k}\,\dvg.$$ The Bochner--Weitzenb\"{o}ck formula yields that
\begin{align*}
I = I_1+I_2:= \int_\M \langle \ric \phi,u\rangle_{\bigwedge^k}\,\dvg + \int_\M \left\langle \widetilde{\Delta}\phi,u\right\rangle_{\bigwedge^k}\,\dvg,
\end{align*}
where $\ric \in {\rm End}\left(\bigwedge^k T^*\M\right)$ depends only on the Riemann curvature tensor of $(\M,g_0)$, and $\widetilde{\Delta}$ is the connection Laplacian. Fix any local orthonormal frame $\{e_i\}_{i=1}^n$. Then one may express (with the summation convention and the abbreviation $\na_i \equiv \na_{e_i}$) that
\begin{equation*}
\widetilde{\Delta}\phi = - \na_i\na_i \phi + \na_{\na_i e_i}\phi.
\end{equation*}
Thus $I_2 = I_{21}+I_{22}+I_{23}+I_{24}$, where
\begin{eqnarray*}
&&I_{21} = -\int_\M \na_i\left(\langle \na_i\phi,u\rangle_{\bigwedge^k}\right)\,\dvg,\\
&&I_{22} = \int_\M \langle \na_i\phi, \na_i u\rangle_{\bigwedge^k}\,\dvg,\\
&&I_{23} = \int_\M\na_{\na_i e_i} \left(\langle \phi,u \rangle_{\bigwedge^k}\right) \,\dvg,\\
&&I_{24} = -\int_\M \left\langle \phi,\na_{\na_i e_i}u\right\rangle_{\bigwedge^k}\,\dvg.
\end{eqnarray*}

Let us treat the four terms $I_{2j}$ one by one. First, 
\begin{align*}
I_{21} = - \sum_{\beta\,\text{ tangential}}\left(\int_\M \na_\beta \langle\na_\beta\phi,u\rangle_{\bigwedge^k}\,\dvg\right) - \int_\M \na_\nu\left(\langle\na_\nu\phi,u\rangle_{\bigwedge^k}\right)\,\dvg.
\end{align*}
The tangential term  is zero, since by the Stokes' or the Gauss--Green theorem one has
\begin{align*}
\int_\M \na_\beta \langle\na_\beta\phi,u\rangle_{\bigwedge^k}\,\dvg = \int_{\p\M} \langle\na_\beta\phi,u\rangle_{\bigwedge^k} \langle \nu, e_\beta\rangle_{\bigwedge^1}\,\dvg=0.
\end{align*}
Similarly,
\begin{align*}
I_{23} &= \int_{\p\M} \langle\phi,u\rangle_{\bigwedge^k} \langle\na_i e_i,\nu\rangle_{\bigwedge^1}\,\dvg\\ &= \int_{\p\M} \G^\nu_{ii} \langle\phi,u\rangle_{\bigwedge^k}\,\dvg,
\end{align*}
where we adopt the notation $\G^{\nu}_{ii}=\G^\ell_{ii}\langle\nu,e_\ell\rangle$; \emph{i.e.}, we shall intentionally confuse $\nu$ with $\nu^\flat$ (and the $1$-form canonically dual to it). Moreover, the definition of the covariant derivative gives us
\begin{align*}
I_{22} = \int_\M \langle\na\phi,\na u\rangle_{\bigwedge^{k+1}}\,\dvg.
\end{align*}
In summary, we have obtained
\begin{align*}
I_2 &= \int_\M \langle\na\phi,\na u\rangle_{\bigwedge^{k+1}}\,\dvg - \int_\M \na_\nu\left(\langle\na_\nu\phi,u\rangle_{\bigwedge^k}\right)\,\dvg \nonumber\\
&\quad+  \int_{\p\M} \G^\nu_{ii} \langle\phi,u\rangle_{\bigwedge^k}\,\dvg  -\int_\M \left\langle \phi,\na_{\na_i e_i}u\right\rangle_{\bigwedge^k}\,\dvg.
\end{align*}

The computation for the boundary term 
\begin{align*}
J:= - \int_{\p\M} \ttt(\delta \phi) \wedge \star \nnn u
\end{align*}
can be performed as in p.64, \cite{s}. Indeed, following the arguments therein, one may infer that
\begin{align*}
J &= -\int_{\p\M} \langle\delta\phi,\nu\mres u\rangle_{\bigwedge^{k-1}}\,\dd\Sigma \\
&= -\int_{\p\M}\left\langle \slashed{\delta}\phi,\nu\mres u\right\rangle_{\bigwedge^{k-1}}\,\dd\Sigma + \int_{\p\M}\left\langle\nu\mres\na_\nu\phi,\nu\mres u\right\rangle_{\bigwedge^{k-1}}\,\dd\Sigma \\
&=:J_1+J_2.
\end{align*}
Here $\slashed{\delta}$ is the codifferential \emph{on the boundary $\p\M$}. 

Notice that $J_1$ is a good term, in the sense that it can be expressed in a way such that no derivatives of $\phi$ or $u$ are involved. To see this, take any shuffle $\sigma$ of tangential indices; \emph{i.e.}, $\sigma$ is in the symmetry group on $\{1,2,\ldots,n\}$, respects the increasing order of indices, and leaves the normal index $\nu$ invariant. Thanks to the Leibniz rule for covariant derivatives, we get
\begin{align*}
&\left[\slashed{\delta}\phi\right]\left(e_{\sigma(1)},\ldots,e_{\sigma(k-1)}\right)\nonumber\\
&= -\sum_{\beta \neq \nu} \na_\beta\phi\left(e_j,e_{\sigma(1)},\ldots,e_{\sigma(k-1)}\right)\nonumber\\
& = -\sum_{\beta \neq \nu}\na_\beta\left\{\phi\left(e_\beta,e_{\sigma(1)},\ldots,e_{\sigma(k-1)}\right)\right\} + \sum_{\beta \neq \nu}\phi\left(\na_\beta e_\beta,e_{\sigma(1)},\ldots,e_{\sigma(k-1)}\right)\nonumber\\
&\qquad+ \sum_{\beta \neq \nu}\phi\left(e_\beta,e_{\sigma(1)},\ldots,\na_\beta e_{\sigma(k)},\ldots,e_{\sigma(k-1)}\right)
\end{align*}
As $\phi$ is purely perpendicular at the boundary, so is $\na_\beta\phi$ for any tangential index $\beta \neq \nu$. Thus, the first term on the right-hand side of the final equality vanishes constantly on $\p\M$. The other two terms contain no derivative of $\phi$ --- therein, the covariant derivatives fall on $e_\gamma$ (for tangential $\gamma$), hence can be expressed in terms of the Christoffel symbols.

Putting the above arguments together, we have deduced from $\mathscr{D}(u,\phi)=I+J$ the following expression:
\begin{align}\label{polarise 1}
&\int_\M \langle du, d\phi\rangle_{\bigwedge^{k+1}} \,\dvg + \int_\M \langle \delta u, \delta \phi\rangle_{\bigwedge^{k-1}} \,\dvg  -  \int_\M \langle\na\phi,\na u\rangle_{\bigwedge^{k+1}}\,\dvg\nonumber\\
&\quad = - \int_\M \na_\nu\left(\langle\na_\nu\phi,u\rangle_{\bigwedge^k}\right)\,\dvg +  \int_{\p\M} \G^\nu_{ii} \langle\phi,u\rangle_{\bigwedge^k}\,\dvg  -\int_\M \left\langle \phi,\na_{\na_i e_i}u\right\rangle_{\bigwedge^k}\,\dvg\nonumber\\
&\qquad+\int_{\p\M}\left\langle\nu\mres\na_\nu\phi,\nu\mres u\right\rangle_{\bigwedge^{k-1}}\,\dd\Sigma - \int_{\p\M} \left\langle[{\rm GOOD}\{\phi\}],\nu\mres u\right\rangle_{\bigwedge^{k-1}}\,\dd\Sigma,
\end{align}
where $[{\rm GOOD}\{\phi\}]$ is the good term arising from $J_1$ as above. (The punchline is that the last term contains no derivative of $u$ on the boundary.)

Now, observe that the left-hand side of Eq.~\eqref{polarise 1} is symmetric in $u$ and $\phi$. By interchanging $u$ and $\phi$, we get
\begin{align}\label{polarise 2}
&\int_\M \langle du, d\phi\rangle_{\bigwedge^{k+1}} \,\dvg + \int_\M \langle \delta u, \delta \phi\rangle_{\bigwedge^{k-1}} \,\dvg  -  \int_\M \langle\na\phi,\na u\rangle_{\bigwedge^{k+1}}\,\dvg\nonumber\\
&\quad = - \int_\M \na_\nu\left(\langle\na_\nu u,\phi\rangle_{\bigwedge^k}\right)\,\dvg +  \int_{\p\M} \G^\nu_{ii} \langle\phi,u\rangle_{\bigwedge^k}\,\dvg  -\int_\M \left\langle u,\na_{\na_i e_i}\phi\right\rangle_{\bigwedge^k}\,\dvg\nonumber\\
&\qquad+\int_{\p\M}\left\langle\nu\mres\na_\nu u,\nu\mres \phi\right\rangle_{\bigwedge^{k-1}}\,\dd\Sigma - \int_{\p\M} \left\langle[{\rm GOOD}\{u\}],\nu\mres \phi\right\rangle_{\bigwedge^{k-1}}\,\dd\Sigma.
\end{align}
Adding up Eqs.~\eqref{polarise 1} and \eqref{polarise 2}, we can take advantage of the metric-compatibility of $\langle\bullet,\bullet\rangle_{\bigwedge^p}$ to deduce that
\begin{align}\label{polarise 3}
&\int_\M \langle du, d\phi\rangle_{\bigwedge^{k+1}} \,\dvg + \int_\M \langle \delta u, \delta \phi\rangle_{\bigwedge^{k-1}} \,\dvg  -  \int_\M \langle\na\phi,\na u\rangle_{\bigwedge^{k+1}}\,\dvg\nonumber\\
&\quad = - \int_\M \na_\nu \na_\nu \langle u,\phi\rangle_{\bigwedge^k}\,\dvg  +  \int_{\p\M} \G^\nu_{ii} \langle\phi,u\rangle_{\bigwedge^k}\,\dvg - \int_\M \na_{\na_i e_i}\langle u,\phi\rangle_{\bigwedge^k}\,\dvg\nonumber\\
&\qquad -\int_{\p\M} \left\langle[{\rm GOOD}\{\phi\}],\nu\mres u\right\rangle_{\bigwedge^{k-1}}\,\dd\Sigma -\int_{\p\M} \left\langle[{\rm GOOD}\{u\}],\nu\mres \phi\right\rangle_{\bigwedge^{k-1}}\,\dd\Sigma\nonumber\\
&\qquad +\int_{\p\M}\left\langle\nu\mres\na_\nu u,\nu\mres \phi\right\rangle_{\bigwedge^{k-1}}\,\dd\Sigma+\int_{\p\M}\left\langle\nu\mres\na_\nu \phi,\nu\mres u\right\rangle_{\bigwedge^{k-1}}\,\dd\Sigma.
\end{align}
The fourth and the fifth terms on the right-hand side of Eq.~\eqref{polarise 3} contain no derivatives. Also, using the same computation for $I_{23}$ as in the above, we find that the second and the third terms cancel each other. 

Finally, let us deal with the last two terms on the right-hand side of Eq.~\eqref{polarise 3}. We proceed by computations in local co-ordinates. Indeed, since $\ttt u =0$, on the boundary we have
\begin{align*}
u = \hat{\sum} u_{\alpha_1\cdots \alpha_{k-1}\nu}dx^{\alpha_1} \wedge \cdots \wedge dx^{\alpha_{k-1}}\wedge \nu,
\end{align*}
where the summation $\hat{\sum}$ is taken over all of $(\alpha_1, \ldots, \alpha_{k-1})$, the increasing $(k-1)$-tuples of tangential indices (\emph{i.e.}, $\alpha_j \neq \nu$). Once again we adopt the abuse of notation $\nu \equiv \nu^\flat$. Then 
\begin{align*}
\na_\nu u &= \hat{\sum} \Bigg(\frac{\p  u_{\alpha_1\cdots \alpha_{k-1}\nu}}{\p\nu}dx^{\alpha_1} \wedge \cdots \wedge dx^{\alpha_{k-1}}\wedge \nu \nonumber\\
&\quad+ \sum_{j=1}^{k-1}u_{\alpha_1\cdots \alpha_{k-1}\nu} dx^{\alpha_1}\wedge\cdots\wedge \na_{\nu}dx^{\alpha_j}\wedge \cdots \wedge dx^{\alpha_{k-1}} \wedge \nu\nonumber\\
&\qquad+ u_{\alpha_1\cdots \alpha_{k-1}\nu}dx^{\alpha_1} \wedge \cdots \wedge dx^{\alpha_{k-1}}\wedge \na_\nu\nu  \Bigg).
\end{align*}
Again, let us write
\begin{align*}
\na_\nu u &= \hat{\sum} \e\frac{\p  u_{\alpha_1\cdots \alpha_{k-1}\nu}}{\p\nu}dx^{\alpha_1} \wedge \cdots \wedge dx^{\alpha_{k-1}}\wedge \nu + \left[{\rm Good}'\{u\}\right]\qquad \text{ on } \p\M,
\end{align*}
where $\left[{\rm Good}'\{u\}\right]$ contains the Christoffel symbols of $(\M,g_0)$ and no derivatives of $u$, and $\e \in \{\pm 1\}$ is a sign. Similarly, by expressing $\phi$ in the same local co-ordinates:
\begin{align*}
\phi = \hat{\sum} \phi_{\alpha_1\cdots \alpha_{k-1}\nu}dx^{\alpha_1} \wedge \cdots \wedge dx^{\alpha_{k-1}}\wedge \nu,
\end{align*}
we can easily compute its contraction with $\nu$:
\begin{align*}
\nu \mres \phi = \hat{\sum} \e \phi_{\alpha_1\cdots \alpha_{k-1}\nu}dx^{\alpha_1} \wedge \cdots \wedge dx^{\alpha_{k-1}}.
\end{align*}
Here $\e$ is the same sign as before. So the penultimate term in Eq.~\eqref{polarise 3} becomes
\begin{align*}
\int_{\p\M}\left\langle\nu\mres\na_\nu u,\nu\mres \phi\right\rangle_{\bigwedge^{k-1}}\,\dd\Sigma &=  \int_{\p\M}\left\{\hat{\sum} \left(\frac{\p  u_{\alpha_1\cdots \alpha_{k-1}\nu}}{\p\nu}\right)\phi_{\alpha_1\cdots \alpha_{k-1}\nu}\right\}\,\dd\Sigma\\
&\quad + \int_{\p\M} \left\langle\left[{\rm GOOD}'\{u\}\right],\nu\mres \phi\right\rangle_{\bigwedge^{k-1}}\,\dd\Sigma.
\end{align*}
By symmetry, the final term in Eq.~\eqref{polarise 3} satisfies
\begin{align*}
\int_{\p\M}\left\langle\nu\mres\na_\nu \phi,\nu\mres u\right\rangle_{\bigwedge^{k-1}}\,\dd\Sigma &= \int_{\p\M}\left\{ \hat{\sum} \left(\frac{\p  \phi_{\alpha_1\cdots \alpha_{k-1}\nu}}{\p\nu}\right)u_{\alpha_1\cdots \alpha_{k-1}\nu}\right\}\,\dd\Sigma\\
&\quad + \int_{\p\M} \left\langle\left[{\rm GOOD}'\{\phi\}\right],\nu\mres u\right\rangle_{\bigwedge^{k-1}}\,\dd\Sigma.
\end{align*}
Adding up together the above two equalities, we get 
\begin{align}\label{polarise, 4}
&\int_{\p\M}\left\langle\nu\mres\na_\nu u,\nu\mres \phi\right\rangle_{\bigwedge^{k-1}}\,\dd\Sigma+\int_{\p\M}\left\langle\nu\mres\na_\nu \phi,\nu\mres u\right\rangle_{\bigwedge^{k-1}}\,\dd\Sigma\nonumber\\
&\quad = \int_{\p\M} \na_\nu \langle u,\phi\rangle_{\bigwedge^k}\,\dd\Sigma + \int_{\p\M} \left\langle\left[{\rm GOOD}'\{u\}\right],\nu\mres \phi\right\rangle_{\bigwedge^{k-1}}\,\dd\Sigma \nonumber\\
&\qquad +\int_{\p\M} \left\langle\left[{\rm GOOD}'\{\phi\}\right],\nu\mres u\right\rangle_{\bigwedge^{k-1}}\,\dd\Sigma.
\end{align}
One more application of the Stokes'/Gauss--Green theorem leads to
\begin{align*}
\int_{\p\M} \na_\nu \langle u,\phi\rangle_{\bigwedge^k}\,\dd\Sigma = \int_\M \na_\nu\na_\nu\langle u,\phi\rangle_{\bigwedge^k}\,\dvg,
\end{align*}
which cancels the first term on the right-hand side of Eq.~\eqref{polarise 3}.

Therefore, summarising Eqs.~\eqref{polarise 3} and \eqref{polarise, 4}, we arrive at \begin{align}\label{polarise 5}
&\int_\M \langle du, d\phi\rangle_{\bigwedge^{k+1}} \,\dvg + \int_\M \langle \delta u, \delta \phi\rangle_{\bigwedge^{k-1}} \,\dvg  -  \int_\M \langle\na\phi,\na u\rangle_{\bigwedge^{k+1}}\,\dvg\nonumber\\
&\quad = \int_{\p\M} \left\langle\left[{\rm GOOD}''\{u\}\right],\nu\mres \phi\right\rangle_{\bigwedge^{k-1}}\,\dd\Sigma +\int_{\p\M} \left\langle\left[{\rm GOOD}''\{\phi\}\right],\nu\mres u\right\rangle_{\bigwedge^{k-1}}\,\dd\Sigma \nonumber\\
&\quad =:\mathscr{G}\{u,\phi\},
\end{align}
where $[{\rm GOOD}''\{\bullet\}]$ involves no derivative of the variable $\bullet$ and depends  only on the $C^1$-geometry of $(\M,g_0)$. Thus, in light of the H\"{o}lder and trace inequalities, the term $\mathscr{G}\{u,\phi\}$ --- which is clearly symmetric in $u$ and $\phi$ --- can be estimated by 
\begin{align}\label{gaffney, final estimate}
|\mathscr{G}\{u,\phi\}|&\leq C\|u\|_{L^r(\p\M)} \|\phi\|_{L^{r'}(\p\M)}\nonumber\\
&\leq C \left(\delta \|\na u\|_{L^r(\M)} + C_\delta \|u\|_{L^r(\M)} \right)\left(\|\na \phi\|_{L^{r'}(\M)} + \|\phi\|_{L^{r'}(\M)} \right),
\end{align}
where $C$ depends only on the $C^1$-geometry of $\M$, $\delta$ is an arbitrary small number, and $C_\delta$ is a large number depending only on $\delta$.

Finally, by choosing $\delta>0$ to be sufficiently small and varying $\phi$ through all the test $k$-forms under the designated boundary condition, we estimate from Eqs.~\eqref{gaffney, final estimate} and \eqref{polarise 5} that
\begin{align*}
&\left|\int_\M \langle\na\phi,\na u\rangle_{\bigwedge^{k+1}}\,\dvg\right| \leq C\left\{ \left(\|d u\|_{L^r(\M)} + \|\delta u\|_{L^r(\M)}+ \|u\|_{L^r(\M)}\right)\|\phi\|_{W^{1,r'}(\M)}\right\},
\end{align*}
noting that $\|d\phi\|_{L^{r'}(\M)}+\|\delta\phi\|_{L^{r'}(\M)}\lesssim \|\phi\|_{W^{1,r'}(\M)}$. Thus the assertion immediately follows from Corollary~\ref{cor: main}.   \end{proof}


\begin{remark}
In view of Theorem~A, we can  interchangeably use the variational quantity 
$$\sup\left\{\frac{\left| ( d u, d \phi)_{g_0} +( \delta u, \delta \phi)_{g_0} + (u,\phi)_{g_0} \right|}{\|d\phi\|_{L^{r'}(\M,g_0)} + \|\delta\phi\|_{L^{r'}(\M,g_0)}+\|\phi\|_{L^{r'}(\M,g_0)}}:\,\phi \in C^\infty_{\diamondsuit}\Omega^k(\M) \right\}$$ in Theorem~B with either of the following three:
\begin{eqnarray*}
&&\sup\left\{\frac{\left| (\na u, \na \phi)_{g_0} + (u,\phi)_{g_0} \right|}{\|d\phi\|_{L^{r'}(\M,g_0)} + \|\delta\phi\|_{L^{r'}(\M,g_0)}+\|\phi\|_{L^{r'}(\M,g_0)}}:\,\phi \in C^\infty_{\diamondsuit}\Omega^k(\M) \right\},\\
&&\sup\left\{\frac{\left| (d u, d \phi)_{g_0} +( \delta u, \delta \phi)_{g_0} + (u,\phi)_{g_0} \right|}{\|\na\phi\|_{L^{r'}(\M,g_0)}+\|\phi\|_{L^{r'}(\M,g_0)}}:\,\phi \in  C^\infty_{\diamondsuit}\Omega^k(\M)\right\},\\
&&\sup\left\{\frac{\left| (\na u, \na \phi)_{g_0} + (u,\phi)_{g_0} \right|}{\|\na\phi\|_{L^{r'}(\M,g_0)}+\|\phi\|_{L^{r'}(\M,g_0)}}:\,\phi \in  C^\infty_{\diamondsuit}\Omega^k(\M) \right\}.
\end{eqnarray*}
\end{remark}

\section{Concluding Remarks}\label{sec: rem}

The Gaffney's inequality is the key estimate for the Hodge decomposition theorem for differential forms in boundary value Sobolev spaces. With Theorem~A established, we can easily deduce the Hodge decomposition from standard functional analysis. See the exposition by Schwarz \cite{s} and Iwaniec--Scott--Stroffolini \cite{iss}, as well as the classical papers of Helmholtz \cite{he}, Weyl \cite{w}, Hodge \cite{hodge}, Kodaira \cite{k}, Duff--Spencer \cite{ds}, Friedrichs \cite{f}, Morrey \cite{m, m'}, and the recent contributions by Dodziuk \cite{d}, Borchers--Sohr \cite{bs}, Simader--Sohr \cite{ss}, Wahl \cite{v}, etc. 

Moreover, one should remark that all the results in this papers remain valid for manifolds-with-boundary $(\M,g_0)$ with $C^1$-bounded geometry.

Our arguments in \S\ref{sec: gaffney} circumvent the local computations in  Euclidean space via the so-called Bochner's technique. Very roughly speaking, the Bochner's technique seem to be more suitable for the $L^2$-setting. In this note,  nevertheless, we have successfully applied it in  non-Hilbert settings. This is precisely due to the variational characterisation of the $W^{1,r}$-norm of differential forms in Theorem~B \emph{\`{a} la} Kozono--Yanagisawa \cite{ky}.

\section*{Acknowledgements}
We are deeply indebted to Gyula Csat\'{o} for kind communications and discussions on recent  developments concerning the Gaffney's inequality.

\end{document}